\documentclass[12pt]{article}
\usepackage{amsmath,amssymb,amsthm,hyperref}
\usepackage{graphicx}
\usepackage{bbm}
\usepackage{natbib}
\usepackage[margin=1.2in]{geometry}
\usepackage{verbatim}
\usepackage{color}
\usepackage[font=small,labelfont=bf]{caption}

\usepackage{tikz}
\usetikzlibrary{decorations.pathreplacing}

\newcommand*\samethanks[1][\value{footnote}]{\footnotemark[#1]}

\title{The bias of isotonic regression}
\author{Ran Dai\thanks{Department of Statistics, University of Chicago}, Hyebin Song\thanks{Department of Statistics, University of Wisconsin--Madison},
 Rina Foygel Barber\samethanks[1], Garvesh Raskutti\samethanks[2]}
\date{\today}

\newtheorem{theorem}{Theorem}
\newtheorem{lemma}{Lemma}

\newtheorem{remark}{Remark}
\newcommand{\R}{\mathbb{R}}
\newcommand{\muh}{\widehat{\mu}}
\newcommand{\EE}[1]{\mathbb{E}\left[{#1}\right]}

\newcommand{\VV}[1]{\mathrm{Var}\left({#1}\right)}
\newcommand{\PP}[1]{\mathbb{P}\left\{{#1}\right\}}

\newcommand{\iso}{\textnormal{iso}}
\newcommand{\eqd}{\stackrel{\textnormal{d}}{=}}
\newcommand{\normal}{\mathcal{N}}

\newcommand{\norm}[1]{\left\|{#1}\right\|}
\newcommand{\One}[1]{\mathbbm{1}\left\{{#1}\right\}}
\DeclareMathOperator{\diag}{diag}
\newcommand{\ident}{\mathbf{I}}
\DeclareMathOperator{\polylog}{polylog}
\newcommand{\iidsim}{\stackrel{\textnormal{iid}}{\sim}}

\begin{document}

\maketitle

\begin{abstract}
We study the bias of the isotonic regression estimator. While there is extensive work characterizing the mean squared error of the isotonic regression estimator,
relatively little is known about the bias. In this paper, we
provide a sharp characterization, proving that the bias scales as $O(n^{-\beta/3})$ up to log factors,
 where $1 \leq \beta \leq 2$ is the exponent corresponding to H{\"o}lder smoothness of the underlying mean. 
Importantly, this result only requires a strictly monotone mean and that the noise distribution has subexponential tails,
without relying on symmetric noise or other restrictive assumptions.
\end{abstract}

\section{Introduction}

\label{sec:intro}

Isotonic regression involves solving a regression problem under monotonicity constraints. In particular, suppose that we observe data $Y\in\R^n$ with $\EE{Y}=\mu$,
where the mean $\mu$ satisfies a monotonicity constraint,
\[\mu_1 \leq \dots\leq \mu_n.\]
To recover the mean vector $\mu$, the least-squares isotonic regression estimator is given by
\[\muh = \iso(Y),\]
where $y\mapsto\iso(y)$ is the projection to the monotonicity constraint, defined on $y\in\R^n$ as
\begin{equation}\label{eqn:define_iso}\iso(y) = \arg\min_{x\in\R^n} \left\{\norm{y-x}^2_2 : x_1\leq \dots \leq x_n\right\}.\end{equation}
Since this constraint defines a convex region in $\R^n$ (in fact, a convex cone),
the resulting optimization problem is convex, and can be solved efficiently
with the ``pool adjacent violators'' (PAVA) algorithm \citep{barlow1972}.

There is a large literature on isotonic regression dating back to the work of \citet{bartholomew59a, bartholomew59b,brunk1955,miles59};
 for further background see, e.g., \citet{barlow1972,robertson1988order}.
Isotonic regression also appears in the problem of estimating a monotone density  estimation, namely, in the
Grenander estimator \citep{grenander1956theory}. 

Recently, there has been renewed interest in isotonic regression as one of the most widely-used examples of regression under shape constraints---for example, the work of \citet{deLeeuw2009,chatterjee2015,gao2017,han2017,guntuboyina18,banerjee2019}. Given the significance of isotonic regression, understanding its statistical properties is of fundamental importance. In this paper, we provide a sharp characterization of the bias of the isotonic regression estimator.

\subsection{Bias and variance}

It is well known that the estimation error $\muh - \mu$ of isotonic regression 
decays at a slower rate than for parametric regression, generally with a dependence
on $n$ that scales as $|\muh_i - \mu_i| \asymp n^{-1/3}$, rather than the rate $n^{-1/2}$ 
that we would expect for parametric problems. 

To make this more precise,
suppose that the mean vector $\mu$ is Lipschitz, satisfying $|\mu_i - \mu_{i+1}|\leq L_1/n$,
and that the noise $Z = Y - \mu$ has independent entries $Z_i$ with $\EE{Z_i}=0$
and with each $Z_i$ assumed to be $\lambda$-subgaussian, that is, $\EE{e^{tZ_i}}\leq e^{t^2\lambda^2/2}$ for any $t\in\R$. In this setting, many works including \citet{brunk1969,cator2011,chatterjee2015}
establish the $n^{-1/3}$ bound on the root-mean-square error of isotonic regression; this scaling can be proved to hold pointwise at each index $i$
(an analogous $n^{-1/3}$ rate is established by \citet{durot2012} for the Grenander estimator of a monotone density).
For instance, \citet{yang2018}
prove that, for any $\delta>0$.
\begin{equation}\label{eqn:yang}\PP{\max_{i_0\leq i \leq n+1-i_0}\big|\muh_i - \mu_i| \leq \sqrt[3]{\frac{8L_1\lambda^2\log\left(\frac{n^2+n}{\delta}\right)}{n}}}\geq 1-\delta,\end{equation}
where
\[i_0 = \left(\frac{\lambda n\sqrt{\log\left(\frac{n^2+n}{\delta}\right)}}{L_1}\right)^{2/3}.\]
Furthermore, \citet{chatterjee2015} establish that this error scaling attains the minimax rate, meaning that we cannot improve on the $n^{-1/3}$
 error rate obtained by fitting $\muh$ via least-squares isotonic regression. However, in certain special cases,
such as when the mean vector $\mu$ is piecewise constant, the error of $\muh$ achieves the parametric $n^{-1/2}$ rate.

While these results summarized above give a fairly complete picture of the tail behavior
of the error $\muh-\mu$ in isotonic regression, relatively little is known about its expected
value $\EE{\muh} - \mu$, i.e., the bias of the least-squares estimator. 
\citet{banerjee2019} establish an asymptotic result under certain smoothness
and strict monotonicity conditions on $\mu$,
proving that the bias is at most
\[\big|\EE{\muh_i} - \mu_i\big| \lesssim n^{-7/15}.\]
In the special case where the variance of the noise is constant across the data,
with $\VV{Z_i} = \sigma^2$ for all $i$, \citet{durot2002} prove an improved bound of 
\[\big|\EE{\muh_i} - \mu_i\big| \lesssim n^{-1/2}.\]
However, as \citet{banerjee2019} point out, in many settings we would prefer to avoid
the assumption of constant variance---for example, if the data is binary with 
$Y_i\sim\textnormal{Bernoulli}(\mu_i)$, the variance of the $Z_i$'s will certainly 
be nonconstant. Furthermore, in the prior literature, it is not clear if $n^{-1/2}$ is the correct scaling for the bias (whether
in the constant-variance or general case), or if this scaling might be possible to improve.

\subsection{Contributions}
In this work, we examine the question of bias more closely, and prove that
\begin{equation}\label{eqn:mainresult_upperbd}
\big|\EE{\muh}-\mu\big|\lesssim \frac{\polylog n}{n^{2/3}},\end{equation}
with no assumption of constant variance, when the underlying mean $\mu$ is Lipschitz, strictly increasing, and smooth.
This is
a faster scaling than what
was previously known for both constant and non-constant variance settings. 
We furthermore establish weaker bounds when $\mu$ satisfies only H{\"o}lder smoothness, with
\begin{equation}\label{eqn:mainresult_upperbd_holder}
\big|\EE{\muh}-\mu\big|\lesssim \left(\frac{\log n}{n}\right)^{\beta/3}\end{equation}
when $\mu$ is H{\"o}lder smooth with exponent $\beta$, for some $1\leq \beta < 2$. ($\beta=2$ corresponds
to the smooth case.)
Matching lower bounds show that, up to log factors, our results are tight.

In particular, if we only assume $\mu$ to be Lipschitz but without smoothness, this corresponds to H{\"o}lder smoothness
with $\beta=1$. In this case, the bias is lower bounded as $n^{-1/3}$ up to log factors. Since it is known
that the error $|\muh_i - \mu_i|$ is bounded on the order of $n^{-1/3}$ (up to logs) with 
high probability, we see that without a smoothness assumption, it may be the case
that the bias is on the same order as the error itself---that is, bias does not vanish relative to variance.
At the other extreme, under smoothness ($\beta = 2$), error scales as $n^{-1/3}$ while bias is bounded as $n^{-2/3}$ (up to logs),
meaning that the bias is vanishing.

\section{Main results}
We begin with some definitions and conditions on the distribution of the data.
We assume that the mean vector $\mu$ is Lipschitz,
\begin{equation}\label{eqn:mu_L1}
|\mu_j - \mu_i|\leq L_1\cdot \frac{j-i}{n}\quad \textnormal{for all $1\leq i\leq j\leq n$}
\end{equation} 
for some $L_1$,
and is monotonic,
\begin{equation}\label{eqn:mu_L0}
\mu_j - \mu_i \geq L_0\cdot \frac{j-i}{n} \quad \textnormal{for all $1\leq i\leq j\leq n$}
\end{equation}
for some $L_0\geq 0$ (note that $L_0>0$ corresponds to a strictly increasing condition). The significance of the strictly increasing assumption is illustrated in Theorem~\ref{thm:lowerbd_strictincr}. In many settings, we might think of the mean values $\mu_i$ as evaluations of an underlying function at a sequence of values---for example an evenly spaced grid, i.e., $\mu_i = f(i/n)$ for $i=1,\dots,n$. In this case, the constants $L_0,L_1$ are lower and upper bounds on the gradient of the underlying function $f$.

Next, we will also assume that $\mu$ is H{\"o}lder smooth, satisfying
\begin{equation}\label{eqn:mu_smooth}
\left| \mu_j  - \left(\frac{k-j}{k-i}\cdot\mu_i + \frac{j-i}{k-i}\cdot\mu_k\right)\right| \leq \frac{M}{4} \cdot \left(\frac{k-i}{n}\right)^{\beta}\quad\textnormal{for all $1\leq i\leq j \leq k\leq n$},
\end{equation}
for some $M$ and some exponent $\beta$ with $1\leq \beta \leq 2$.  As before, if $\mu_i = f(i/n)$ for some underlying function $f$ defining the signal, then $(\beta,M)$-H{\"o}lder smoothness of the function $f$, defined as the property that 
\begin{equation}\label{eqn:function_holder_smooth}|f'(t_0)-f'(t_1)|\leq M|t_0-t_1|^{\beta-1}\textnormal{ for all $t_0,t_1$},\end{equation} is sufficient to ensure that the vector of mean values $\mu$ satisfies this assumption.

The exponent $\beta$ in assumption~\eqref{eqn:mu_smooth} controls the smoothness, with $\beta = 2$ corresponding to a bounded second derivative (or a Lipschitz gradient) while $\beta<2$ denotes a weaker smoothness assumption. In particular, if we were to take $\beta=1$, then this assumption does not in fact imply any smoothness, as it is trivially satisfied with $M = L_1$ for any signal $\mu$ that is $L_1$-Lipschitz as in our condition~\eqref{eqn:mu_L1}.

Next we turn to our assumptions on the noise $Z$. We assume independent noise with subexponential tails:
\begin{multline}\label{eqn:noise_assump}
\textnormal{The $Z_i$'s are independent, with $\EE{Z_i}=0$ }\\
\textnormal{and $\EE{e^{t Z_i}}\leq e^{t^2\lambda^2/2}$ for all $|t|\leq \tau$ and all $i=1,\dots,n$}.
\end{multline}
We will also require that the variances $\sigma^2_i$
are bounded from below and are Lipschitz along the sequence $i=1,\dots,n$, satisfying
\begin{equation}\label{eqn:sig_Lip}
 \VV{Z_i} = \sigma^2_i \geq\sigma^2_{\min}>0\textnormal{ and }
\left|\sigma_i - \sigma_j\right| \leq L_\sigma\cdot \frac{j-i}{n}\quad\textnormal{for all $1\leq i\leq j\leq n$}.
\end{equation}
(Note that we must have $\sigma^2_i\leq \lambda^2$ by the subexponential tails assumption~\eqref{eqn:noise_assump}.)

\subsection{An upper bound for smooth signals}
We now present our upper bound on the bias of isotonic regression.
\begin{theorem}\label{thm:upperbd_smooth}
Let $Y=\mu+Z$, where the signal $\mu$ and noise $Z$ satisfy
 assumptions~\eqref{eqn:mu_L1}, \eqref{eqn:mu_L0}, \eqref{eqn:mu_smooth}, \eqref{eqn:noise_assump}, \eqref{eqn:sig_Lip}, 
 with parameters $0<L_0\leq L_1$, $M$, $\beta\in[1,2]$, $\lambda>0$, $\tau>0$, $L_\sigma>0$, $\sigma_{\min}>0$. 
 Then $\muh=\iso(Y)$ satisfies
\[\big|\EE{\muh_i}-\mu_i\big| \leq \begin{cases} C\left(\frac{\log n}{n}\right)^{\beta/3}, &\textnormal{ if } 1\leq \beta < 2, \\
C\left(\frac{(\log n)^5}{n}\right)^{2/3}, &\textnormal{ if } \beta = 2,\end{cases}
 \]
 for all $i$ with
\[C' (\log n)^{1/3} n^{2/3} \leq i \leq n+1- C'(\log n)^{1/3} n^{2/3},\]
where $C,C'$ depend only on the parameters in our assumptions,
and not on $n$.
\end{theorem}
\noindent We remark that in the case of Gaussian noise, 
the upper bound can be tightened to $C\left(\frac{\log n}{n}\right)^{\beta/3}$ for any $\beta\in [1,2]$, i.e., the extra power of the log term
for the case $\beta=2$ no longer appears (in the non-Gaussian case, it is likely an artifact of the proof).

To better understand the result of this theorem, consider the extreme case where we take $\beta=1$ (which, as mentioned
above, simply reduces to the Lipschitz assumption). 
In this case, the upper bound of Theorem~\ref{thm:upperbd_smooth} results in the bias bound
\[\max_{i_0 \leq i \leq n+1- i_0}\big|\EE{\muh_i}-\mu_i\big| \lesssim \sqrt[3]{\frac{\log n}{n}}.\]
which is, up to a constant, the same as the bound~\eqref{eqn:yang} proved by \citet{yang2018} to hold {\em with high probability} on the maximum entrywise error
$\big|\muh_i - \mu_i\big|$.
 In other words, this suggests that the bias may be as large as the (square root) variance. 
 
 On the other hand, if $\beta=2$, then the bias scales as $\left(\polylog n/n\right)^{2/3}$ while the high probability
  bound on the error is still $\left(\log n /n\right)^{1/3}$---the bias is vanishing relative to the error of any one draw of the data.

\subsubsection{Simulation}
To explore this scaling, we conduct a simple simulation\footnote{Code available at
 \url{http://www.stat.uchicago.edu/~rina/code/iso_bias_simulation_code.R}} to compare the case $\beta=2$ with $\beta=1$.
For these two cases, Theorem~\ref{thm:upperbd_smooth} establishes that,
at each index $i$ (bounded away from the endpoints $1$ and $n$), bias is bounded as $\asymp n^{-2/3}$ and $\asymp n^{-1/3}$, respectively, up to log factors.
We will see that this scaling is achieved by our simple examples.

\begin{figure}[t]
\includegraphics[width=\textwidth]{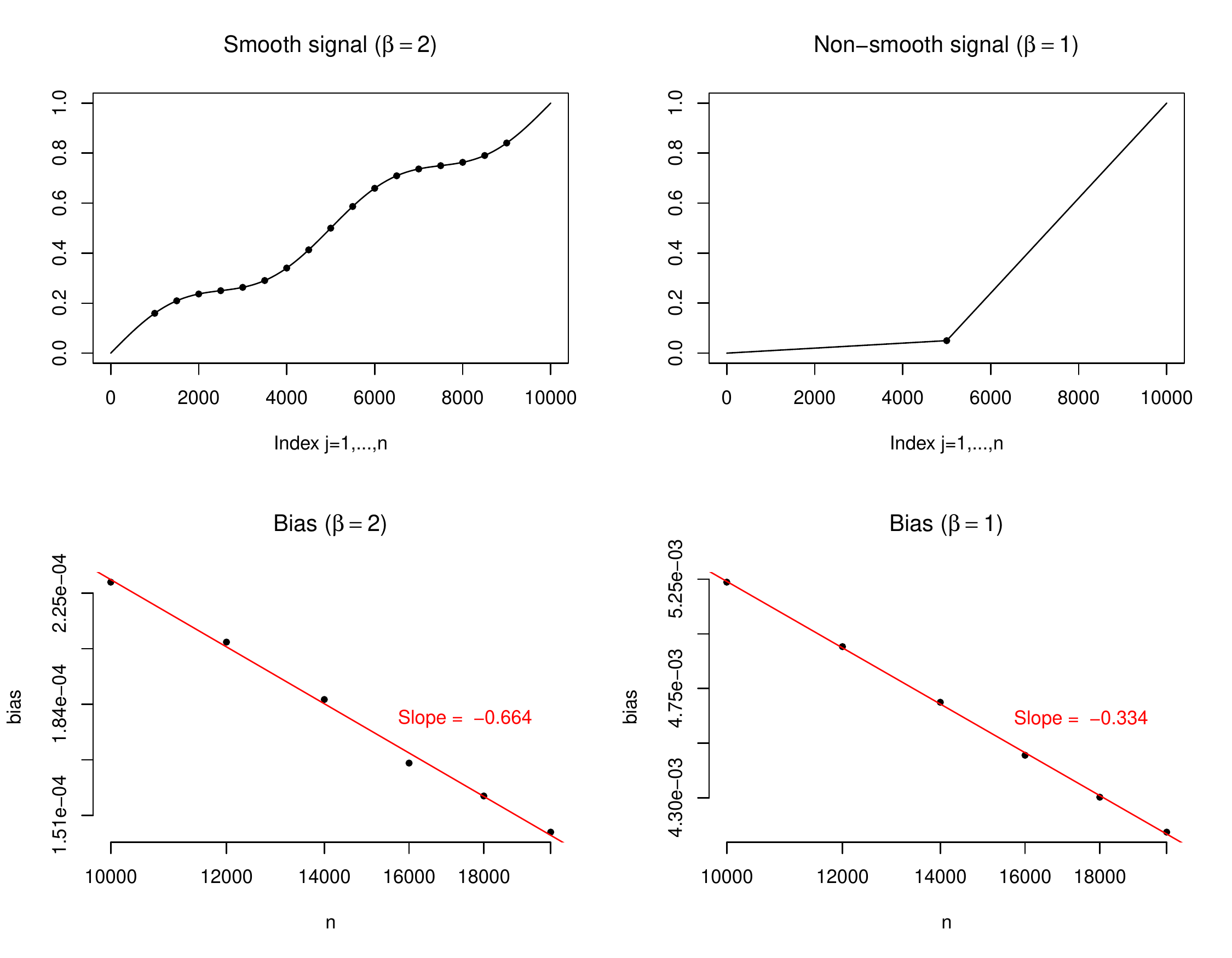}
\caption{Top: illustration of the mean vectors for the simulation, for the case $n=10000$.
The  indices for which we measure the bias are indicated
in each plot.\\
Bottom: simulation results, plotting bias against $n$ (each on the log scale), for the smooth and non-smooth case. The line and slope in each plot are the least-squares
regression line (for $\log(|\textnormal{bias}|)$ regressed on $\log(n)$).}
\label{fig:simulation}
\end{figure}

The two mean vectors, $\mu^{\textnormal{s}}$ for the smooth case and $\mu^{\textnormal{ns}}$ for the non-smooth case,
 are illustrated in the top two panels of Figure~\ref{fig:simulation}.
The smooth mean is defined with a sine wave function, while the non-smooth mean is a piecewise linear ``hinge'' shape:
\[\mu^{\textnormal{s}}_i = (i/n)+\sin(4\pi \cdot i/n)/16,\quad\quad \mu^{\textnormal{ns}}_i = \begin{cases} 0.1(i/n), &  i\leq n/2, \\ 1.9 (i/n) -0.9, &  i>n/2,\end{cases}\]
For both the smooth and non-smooth mean, writing $\mu$ 
to denote either $\mu^{\textnormal{s}}$ or $\mu^{\textnormal{ns}}$ as appropriate,
 we generate a signal $Y = \mu+\normal(0,\sigma^2\ident_n)$ for $\sigma = 0.1$, and compute $\iso(Y)$.
Our final result is the magnitude of the bias at the midpoint for the non-smooth case,
\[\textnormal{bias} =\big|\EE{\iso(Y)_i} - \mu_i\big| \textnormal{ for $i=n/2$},\]
or averaged over several points for the smooth case,
\[\textnormal{bias} =\textnormal{Mean}\Big\{\big|\EE{\iso(Y)_i}- \mu_i\big| : i = 0.1n, 0.15n, 0.2n,\dots,0.9n\Big\},\]
where $\EE{\iso(Y)}$ is estimated by averaging 500,000 trials.
We run the same procedure for each $n=10000,12000,\dots,20000$.
Our simulation results, shown in the bottom panels of Figure~\ref{fig:simulation}, verify that the
bias at the selected points indeed appears to scale as $\asymp n^{-2/3}$ for the smooth case, and as $\asymp n^{-1/3}$ for the non-smooth case.
We next turn to the question of establishing these lower bounds theoretically.

\subsection{A matching lower bound for smoothness}
We next show that, as suggested by our simulations,
 our dependence on the smoothness assumption is tight---up to constants and log factors,
 we cannot improve our dependence on the H{\"o}lder exponent $\beta$ (i.e., the power of $n$, $n^{-\beta/3}$, appearing in our upper bound, Theorem~\ref{thm:upperbd_smooth}).
 While our simulated example only verifies this at a constant number of indices $i$,
here we show a stronger result---the lower bound is attained by a constant {\em fraction} of the indices $i=1,\dots,n$.

\begin{theorem}\label{thm:lowerbd_smoothness}
Fix any parameters $L_1>L_0>0$, $M>0$, $1\leq \beta \leq 2$, and $\sigma^2>0$. Then for any $n\geq C$, 
there exists some $\mu\in\R^n$ that is  $L_1$-Lipschitz~\eqref{eqn:mu_L1}, $L_0$-strictly increasing~\eqref{eqn:mu_L0}, and $(\beta,M)$-smooth~\eqref{eqn:mu_smooth}, such that,
 for $Y\sim \normal(\mu,\sigma^2\mathbf{I}_n)$ and $\muh = \iso(Y)$, the bias satisfies
\[\big|\EE{\muh_i} - \mu_i\big| \geq C' n^{-\beta/3} (\log n)^{-5\beta/3}\]
for at least $C'' n$ many indices $i\in\{1,\dots,n\}$,
where $C,C',C''$ depend only on the parameters in our assumptions,
and not on $n$.
\end{theorem}
\noindent Note that this noise distribution, i.e.~$Z_i\stackrel{\textnormal{iid}}{\sim}\normal(0,\sigma^2)$, trivially satisfies the assumptions~\eqref{eqn:noise_assump} and~\eqref{eqn:sig_Lip} that we require in our upper bound. In other words, the lower bound result shows that our upper bound is tight for all $1\leq \beta \leq 2$, up to log factors. 

\subsection{Necessity of the strictly increasing mean}

Finally, we study the role of the strictly increasing assumption for the mean $\mu$ \eqref{eqn:mu_L0}. 
\citet{wright1981} studied this problem in a different context, considering
a signal $\mu_i = f(i/n)$ where the function $f$ is monotone nondecreasing and satisfies
\begin{equation}\label{eqn:alpha_wright}|f(t) - f(t_0)|\asymp |t-t_0|^{\alpha}\textnormal{ as }t\rightarrow t_0\end{equation}
at some point $t_0\in(0,1)$, for some $\alpha\geq 1$. 
For example, if $\alpha$ is an integer, this is satisfied if $f$ is required to have $f^{(k)}(t_0)=0$
for all $k=1,\dots,\alpha-1$, and $f^{(\alpha)}(t_0)>0$,
where $f^{(k)}$ denotes the $k$th derivative of $f$.
In this setting, writing $t_0 = i_0/n$, \citet[Theorem 1]{wright1981} establishes
that the rescaled error $n^{\alpha/(2\alpha+1)}\big(\muh_{i_0}-\mu_{i_0}\big)$ converges to some fixed distribution.
In other words, the error $\big|\muh_{i_0}-\mu_{i_0}\big|$ 
is on the scale $n^{-\alpha/(2\alpha+1)}$.

\begin{remark}[Smoothness condition versus Wright's condition] The H{\"o}lder smoothness condition~\eqref{eqn:function_holder_smooth}
and Wright's condition~\eqref{eqn:alpha_wright} on functions $f:[0,1]\rightarrow\R$ appear very similar initially, but have different implications. Consider the
case $\alpha = \beta=2$. If the function $f$ is twice differentiable, the H{\"o}lder smoothness condition~\eqref{eqn:function_holder_smooth}
is equivalent to requiring that $|f''(t)|$ is bounded (for all $t$), while Wright's condition~\eqref{eqn:alpha_wright}
instead requires that $|f''(t_0)|$ is bounded and additionally $f'(t_0)=0$. In other words, the H{\"o}lder condition requires that $f$ is smooth,
while Wright's condition requires that $f$ is both smooth and flat. 
\end{remark}

Our next result establishes that, in the worst case, the bias of the isotonic regression estimator is on the same order up to log factors---that is,
for any $\alpha\geq 1$, we construct a signal $\mu$ satisfying \citet{wright1981}'s condition~\eqref{eqn:alpha_wright}
such that $\big|\EE{\muh_{i_0}}-\mu_{i_0}\big|$ scales as $n^{-\alpha/(2\alpha+1)}$ up to log factors.

\begin{theorem}\label{thm:lowerbd_strictincr}

Fix any parameters $L_1>0$, $M>0$, and $\sigma^2>0$. Then for any $n\geq C$, $\alpha \geq 1$, 
there exists some $\mu\in\R^n$ that is $L_1$-Lipschitz~\eqref{eqn:mu_L1}, monotone nondecreasing
(i.e., $L_0$-increasing~\eqref{eqn:mu_L0} with $L_0=0$), 
and is equal to $\mu_i = f(i/n)$ for some function $f$ satisfying the condition~\eqref{eqn:alpha_wright} at $t_0=0.5$,
such that,
 for $Y\sim \normal(\mu,\sigma^2\mathbf{I}_n)$ and $\muh = \iso(Y)$, the bias satisfies
\[\big|\EE{\muh_{i_0}} - \mu_{i_0}\big| \geq C' (n (\log n)^2)^{-\frac{\alpha}{2\alpha+1}}\]
at the index $i_0 = n/2$, i.e., the index satisfying $t_0 = i_0/n$. (Here $C,C'$  depend only on the parameters in our assumptions,
and not on $n$.)
\end{theorem}   
\noindent 
Up to constants and log factors, this result matches the upper bound proved by \citet{wright1981}.

We remark that, for $\alpha\geq 2$, the signal $\mu$ constructed in the
proof of the theorem is $(\beta,M)$-smooth with $\beta=2$.
(For $\alpha<2$, the signal $\mu$ constructed in the proof is instead $(\beta,M)$-smooth with $\beta=\alpha$.)
Setting $\alpha=2$, for example, we obtain a lower bound on bias that scales as $n^{-2/5}$,
while as $\alpha\rightarrow\infty$ the scaling approaches $n^{-1/2}$ (up to log factors).
We can compare this to our results in Theorems~\ref{thm:upperbd_smooth} and~\ref{thm:lowerbd_smoothness},
where we establish a faster scaling $n^{-2/3}$ for the worst-case bias (up to log factors) for signals $\mu$ with H{\"o}lder
exponent $\beta=2$ that are also assumed to be strictly increasing. The gap between these powers of $n$ highlights
the role of the strict monotonicity condition in the isotonic regression problem.

\section{Proofs of main results}

\subsection{Properties of isotonic regression}
Before we proceed with our proofs, it will be useful to recall some well-known properties of the isotonic projection, $y\mapsto\iso(y)$
(for background see, e.g., \citet[Chapter 1]{barlow1972}).

First, the individual entries of $\iso(y)$ can be expressed with the useful ``min-max'' formula \citep[(1.9)]{barlow1972}: for any $y\in\R^n$ and any index $i$,
\begin{equation}\label{eqn:minmax}
\iso(y)_i  = \max_{1\leq j\leq i}\min_{i\leq k\leq n} \overline{y}_{j:k},\end{equation}
where 
\[\overline{y}_{j:k} = \frac{1}{k-j+1}\sum_{\ell=j}^k y_\ell\]
is the average of the subvector $(y_j,\dots,y_k)$ of $y$, and the max and min are attained at $j = j_i(y)$ and $k=k_i(y)$ where
\begin{equation}\label{eqn:minmax_at_jk}
j_i(y) = \min\{j : \iso(y)_j = \iso(y)_i\}\textnormal{ and }k_i(y) = \max\{k : \iso(y)_k = \iso(y)_i\},
\end{equation}
the endpoints of the constant interval of $\iso(y)$ containing index $i$.

Second, using the min-max definition in (\ref{eqn:minmax}), it is trivial to see that isotonic regression commutes with shifts in location and scale: for any $y\in\R^n$, any $a\in\R$, and any $b>0$, it holds that~\citep[(1.14)+(1.15)]{barlow1972}
\begin{equation}\label{eqn:iso_locationscale}\iso(a\cdot \mathbf{1}_n + b\cdot y) = a\cdot \mathbf{1}_n + b\cdot \iso(y).\end{equation}

Next, if we run isotonic regression on a subvector $(y_a,\dots,y_b)$ of $y$, this can only add breakpoints relative to $y$. Specifically,
for any $y\in\R^n$, and any indices $1 \leq a \leq i \leq b \leq n$,
define
\[\tilde{y} = y_{[a:b]} = \big(y_a,y_{a+1},\dots,y_b\big)\in\R^{b-a+1}.\]
Then
\begin{equation}\label{eqn:subsequence_breakpoint}\iso(\tilde{y})_{i-a+1}= \iso(\tilde{y})_{i-a+2} \quad \Rightarrow \quad \iso(y)_i = \iso(y)_{i+1}.\end{equation}
(Note that indices $i-a+1,i-a+2$ in the subvector $\tilde{y}$ correspond to indices $i,i+1$ in the full vector $y$.)
Finally, on the same subvector $\tilde{y}=(y_a,\dots,y_b)$, for any index $i$ with $a\leq i\leq b$, 
\begin{multline}\label{eqn:subsequence_equal}
\textnormal{If $\iso(y)_{a-1}\neq \iso(y)_i$ (or $a=1$) and $\iso(y)_i\neq \iso(y)_{b+1}$ (or $b=n$),}\\\textnormal{ then $\iso(y)_i = \iso(\tilde{y})_{i-a+1}$}.\end{multline}
That is, truncating the sequence $y$ at a breakpoint of $\iso(y)$ will not affect the estimated values. 
These last two properties~\eqref{eqn:subsequence_breakpoint} and~\eqref{eqn:subsequence_equal} follow from
the definition of the ``pool adjacent violators'' (PAVA) algorithm for computing $\iso(y)$ \citep[pg.~13--15]{barlow1972}.
\subsection{Breakpoint lemma}\label{sec:breakpointlemma}

Before proving our theorems, we first present a result bounding the probability of a breakpoint occurring at any particular location in a Gaussian isotonic regression problem.
We will use this lemma to prove both our upper and lower bounds.

\begin{lemma}[Breakpoint lemma]\label{lem:breakpointlemma}
Let $Y\sim \normal\big((\mu_1,\dots,\mu_n),\diag\{\sigma^2_1,\dots,\sigma^2_n\}\big)$. Fix any integer $m\geq 2$ and any index $i$ with $m \leq i \leq n-m$. 
Define
\[\bar{\mu} = \frac{1}{2m} \sum_{j=i-m+1}^{i+m} \mu_j\textnormal{\quad and\quad }\bar{\sigma}^2 = \frac{1}{2m} \sum_{j=i-m+1}^{i+m} \sigma^2_j,\]
and assume that
\begin{equation}\label{eqn:breakpointlemma_nearlyconstant}
\sum_{j=i-m+1}^{i+m} \left(\frac{\mu_j - \bar{\mu}}{\bar{\sigma}}\right)^2 \leq \frac{C_1}{ \log m}\textnormal{\quad and\quad }\max_{i-m+1 \leq j \leq i+m} \left|\frac{\sigma^2_j - \bar{\sigma}^2}{\bar{\sigma}^2}\right|\leq \frac{C_2}{\sqrt{m \log m}}.\end{equation}
Then
\[\PP{\iso(Y)_i \neq \iso(Y)_{i+1}} \leq \frac{C_3\log m}{m},\]
where $C_3$ depends only on $C_1,C_2$ and not on $m$ or $n$.
\end{lemma}
\noindent In other words, if the means $\mu_j$ and variances $\sigma^2_j$ are approximately constant near index $i$, then the probability of a breakpoint
at index $i$ is low.
To gain intuition for the consequences of this lemma,
consider a simple setting where the mean $\mu$ is $L_1$-Lipschitz~\eqref{eqn:mu_L1}, and the variances are constant, $\sigma^2_i=\sigma^2>0$.
In this case, the conditions~\eqref{eqn:breakpointlemma_nearlyconstant} are satisfied for $m\asymp n^{2/3}/(\log n)^{1/3}$,
and the probability of a breakpoint at a given index $i$ with $m\leq i\leq n-m$ is therefore $\lesssim (\log n)^{4/3}/n^{2/3}$. We can then expect constant segments
of $\iso(Y)$ to have length $\gtrsim n^{2/3}/(\log n)^{4/3}$.
If instead the mean is constant with $\mu_1=\dots=\mu_n$ (and the variance is again constant), then the conditions~\eqref{eqn:breakpointlemma_nearlyconstant} are satisfied for $m\asymp n$; in this case we can expect
$\iso(Y)$ to have constant segments of length $\gtrsim n/(\log n)$.

In order to prove this result, we first consider the case that $Y$ is standard Gaussian. We will use the following classical result:
\begin{lemma}[\cite{Andersen_1954}]\label{lem:Andersen}
Let $W\sim\normal(\mathbf{0}_m,\mathbf{I}_m)$ for any $m\geq 1$. Then the number of piecewise constant segments of $\iso(W)$, denoted by $N(W)$, is distributed as
\[N(W)\eqd I_1 + \dots + I_m,\]
where $I_1,\dots,I_m$ are independent Bernoulli random variables with $\PP{I_j = 1} = 1/j$.
\end{lemma}
\noindent This result allows us to prove the breakpoint lemma for a standard Gaussian:
\begin{lemma}\label{lem:breakpoint_standardnormal}
Let $Y\sim\normal\big(\mathbf{0}_{2m},\mathbf{I}_{2m}\big)$ for $m\geq 2$. 
Then
$\PP{\iso(Y)_m<\iso(Y)_{m+1}} \leq \frac{\log m}{m-1}$.
\end{lemma}
\begin{proof}[Proof of Lemma~\ref{lem:breakpoint_standardnormal}]
From Lemma~\ref{lem:Andersen}, for a sequence $W\sim\normal(\mathbf{0}_m,\mathbf{I}_m)$, the expected number of breakpoints in $\iso(W)$ 
is given by
\[\EE{\sum_{i=1}^{m-1}\One{\iso(W)_i < \iso(W)_{i+1}}} = \EE{N(W)-1} \leq \left(1 + \dots + \frac{1}{m}\right) - 1 \leq \log(m).\] 
Therefore, there must be some $i\in\{1,\dots,m-1\}$  such that $\PP{\iso(W)_i < \iso(W)_{i+1}} \leq \frac{\log m}{m-1}$. Next, we let $\tilde{Y} = (\tilde{Y}_1,\dots,\tilde{Y}_m)$ be a subvector of $Y$ with $\tilde{Y}_j=Y_{m-i+j}$ for $j=1,\dots,m$, which again has the distribution $\tilde{Y} \sim \normal(\mathbf{0}_m,\mathbf{I}_m)$. The relationship between $Y$ and $\tilde{Y}$
is illustrated in Figure~\ref{fig:breakpointlemma}. By the property~\eqref{eqn:subsequence_breakpoint} of isotonic regression,  we know that
\[\iso(Y)_m < \iso(Y)_{m+1}  \ \Rightarrow \ \iso(\tilde{Y})_m < \iso(\tilde{Y})_{m+1},\]
which concludes the proof.
\end{proof}
\noindent Finally, the proof of Lemma~\ref{lem:breakpointlemma} for the general case is established by comparing the distribution of $Y$
to the distribution of a standard Gaussian random vector, using our assumptions on the low variability in the
$\mu_j$'s and $\sigma_j$'s near the index $i$.
The proof is given in Appendix~\ref{sec:proof_breakpointlemma}.

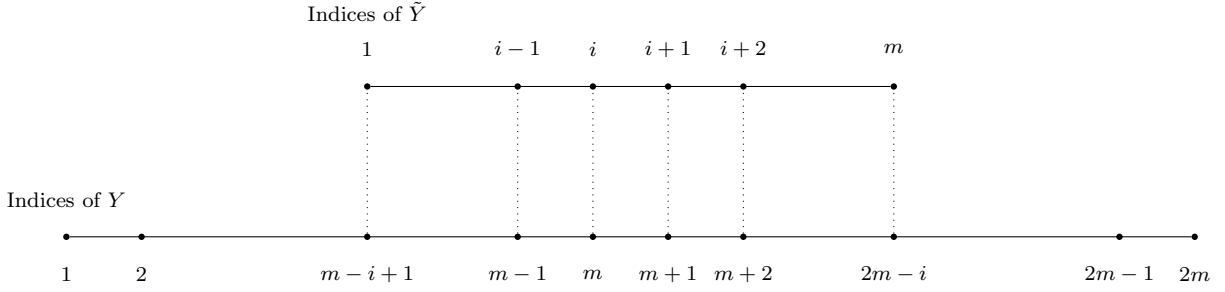
\begin{figure}[t]
\scriptsize
\begin{tikzpicture}
\begin{scope}[xshift=0cm,yshift=0cm]

\node[circle,fill=black,scale=0.3] (y1) at (1,5) {};
\node[circle,fill=black,scale=0.3] (y2) at (2,5) {};
\node[circle,fill=black,scale=0.3] (ymminusiplus1) at (5,5) {};
\node[circle,fill=black,scale=0.3] (ymminus1) at (7,5) {};
\node[circle,fill=black,scale=0.3] (ym) at (8,5) {};
\node[circle,fill=black,scale=0.3] (ymplus1) at (9,5) {};
\node[circle,fill=black,scale=0.3] (ymplus2) at (10,5) {};
\node[circle,fill=black,scale=0.3] (y2mminusi) at (12,5) {};
\node[circle,fill=black,scale=0.3] (y2mminus1) at (15,5) {};
\node[circle,fill=black,scale=0.3] (y2m) at (16,5) {};

\node[circle,fill=black,scale=0.3] (yt1) at (5,7) {};
\node[circle,fill=black,scale=0.3] (ytiminus1) at (7,7) {};
\node[circle,fill=black,scale=0.3] (yti) at (8,7) {};
\node[circle,fill=black,scale=0.3] (ytiplus1) at (9,7) {};
\node[circle,fill=black,scale=0.3] (ytiplus2) at (10,7) {};
\node[circle,fill=black,scale=0.3] (ytm) at (12,7) {};

\node at (1,4.5) {$1$};
\node at (2,4.5) {$2$};
\node at (5,4.5) {$m-i+1$};
\node at (7,4.5) {$m-1$};
\node at (8,4.5) {$m$};
\node at (9,4.5) {$m+1$};
\node at (10,4.5) {$m+2$};
\node at (12,4.5) {$2m-i$};
\node at (15,4.5) {$2m-1$};
\node at (16,4.5) {$2m$};

\node at (1,5.5) {Indices of $Y$};

\node at (5,7.5) {$1$};
\node at (7,7.5) {$i-1$};
\node at (8,7.5) {$i$};
\node at (9,7.5) {$i+1$};
\node at (10,7.5) {$i+2$};
\node at (12,7.5) {$m$};

\node at (5,8) {Indices of $\tilde{Y}$};

\draw[dotted] (ymminusiplus1.north) -- (yt1.south);
\draw[dotted] (ymminus1.north) -- (ytiminus1.south);
\draw[dotted] (ym.north) -- (yti.south);
\draw[dotted] (ymplus1.north) -- (ytiplus1.south);
\draw[dotted] (ymplus2.north) -- (ytiplus2.south);
\draw[dotted] (y2mminusi.north) -- (ytm.south);

\draw[-] (y1.east) -- (y2m.west);
\draw[-] (yt1.east) -- (ytm.west);
\end{scope}
\end{tikzpicture}
\caption{Illustration of the proof idea for Lemma~\ref{lem:breakpoint_standardnormal}.
A break between indices $i$ and $i+1$ in the isotonic regression of $\tilde{Y}$, corresponds to a break between indices $m$ and $m+1$ for $Y$.}
\label{fig:breakpointlemma}
\end{figure}

\subsection{Proof of upper bound (Theorem~\ref{thm:upperbd_smooth})}\label{sec:proof_upperbd}
The proof of our upper bound will follow these steps to bound the bias of $\iso(Y)_i$:
\begin{itemize}
\item {\bf Step 1:} We replace the random vector $Y$ with a Gaussian random vector $\tilde{Y}$ whose
entries have the same means and variances, and 
bound the change in the bias at index $i$.
\item {\bf Step 2:} From the Gaussian random vector $\tilde{Y}$, we extract a subvector of length $\asymp n^{2/3}(\log n)^{1/3}$ centered at index $i$,
and 
bound the change in the bias at index $i$.
\item {\bf Step 3:} We take an approximation to the new subvector, with linearly increasing means and with constant variance,
and bound the change in the bias at index $i$. We will also see that the new approximation has zero bias due to symmetry.
\end{itemize}
\subsubsection{Step 1: reduce to the Gaussian case}
We will first reduce the general problem to a Gaussian approximation, using the following lemma:
\begin{lemma}\label{lem:coupling}
Fix $n\geq 2$, and suppose $Y\in\R^n$ is a random vector satisfying assumptions~\eqref{eqn:mu_L1},~\eqref{eqn:noise_assump}, and~\eqref{eqn:sig_Lip} 
with parameters $L_1>0$, $\lambda>0$, $\tau>0$, $L_\sigma>0$, and $\sigma_{\min}>0$.
Then
there exists a coupling between $Y$ and
\[\tilde{Y}\sim\normal\big((\mu_1,\dots,\mu_n),\diag\{\sigma^2_1,\dots,\sigma^2_n\}\big)\]
satisfying
\[\EE{\big|\iso(Y)_i - \iso(\tilde{Y})_i\big|} \leq \frac{C_1(\log n)^{10/3}}{n^{2/3}}\textnormal{ for all $i$ with $\frac{C_2n^{2/3}}{(\log n)^{1/3}} \leq i \leq n - \frac{C_2 n^{2/3}}{(\log n)^{1/3}}$,}\]
where $C_1,C_2$ depend only on $L_1,\lambda,\tau,L_{\sigma},\sigma_{\min}$, and not on $n$.
\end{lemma}
\noindent This result is proved in Appendix~\ref{sec:coupling_proof}, using
the Gaussian coupling theorem of \citet[Theorem 1]{sakhanenko1985convergence} along with our breakpoint lemma (Lemma~\ref{lem:breakpointlemma}).

Now define $\tilde{Y} \sim\normal\big((\mu_1,\dots,\mu_n),\diag\{\sigma^2_1,\dots,\sigma^2_n\}\big)$.
Applying Lemma~\ref{lem:coupling} together with the triangle inequality, we see that
\begin{equation}\label{eqn:step1_bias}\left|\EE{\iso(Y)_i} - \mu_i\right|  \leq \left|\EE{\iso(\tilde{Y})_i} - \mu_i\right| +  \frac{C_1(\log n)^{10/3}}{n^{2/3}}\end{equation}
for all $i$ with 
\begin{equation}\label{eqn:step1_i}\frac{C_2 n^{2/3}}{(\log n)^{1/3}} \leq i \leq n - \frac{C_2 n^{2/3}}{(\log n)^{1/3}},
\end{equation}
 for some $C_1,C_2$ depending on $L_1,\lambda,\tau,L_\sigma,\sigma_{\min}$.

From this point on, then, it suffices to bound the first term in this upper bound---that is, we 
need to prove the main result, Theorem~\ref{thm:upperbd_smooth}, in the special case that
the noise terms $Z_i$ are Gaussian. From this point on, we will work with $\tilde{Y} = \mu + \tilde{Z}$ where $\tilde{Z}_i\sim\normal(0,\sigma^2_i)$.

\subsubsection{Step 2: extract a subvector}
Next, let
\[m = \left\lceil\left(\frac{18\lambda n\sqrt{L_1\log n}}{L_0^{3/2}}\right)^{2/3}\right\rceil - 1\]
and assume that 
$m\leq i\leq n+1-m$.
We will
see that $\EE{\iso(\tilde{Y})_i}$ is not changed substantially if we
instead calculate the isotonic regression of
\[\tilde{Y}^{(i)} = (\tilde{Y}_{i-m},\dots,\tilde{Y}_i,\dots,\tilde{Y}_{i+m})\in\R^{2m+1}.\]
Note that index $m+1$ in the subvector $\tilde{Y}^{(i)}$ corresponds to index $i$ in the full vector $\tilde{Y}$.

Now we bound the bias of $\iso(\tilde{Y})_i$ in terms of the subvector. We can write
\begin{align*}
&\left|\EE{\iso(\tilde{Y})_i} - \EE{\iso(\tilde{Y}^{(i)})_{m+1}}\right| \\
& \leq \EE{\left|\iso(\tilde{Y})_i - \iso(\tilde{Y}^{(i)})_{m+1}\right| } \\
&= \EE{\left|\iso(\tilde{Y})_i - \iso(\tilde{Y}^{(i)})_{m+1}\right| \cdot \One{\iso(\tilde{Y})_i\neq \iso(\tilde{Y}^{(i)})_{m+1}}}\\ 
&\leq \sqrt{\EE{\left(\iso(\tilde{Y})_i - \iso(\tilde{Y}^{(i)})_{m+1}\right)^2 }}\cdot\sqrt{\PP{\iso(\tilde{Y})_i\neq \iso(\tilde{Y}^{(i)})_{m+1}}} .\end{align*}
Deterministically, we have\begin{multline*}
\left|\iso(\tilde{Y})_i - \iso(\tilde{Y}^{(i)})_{m+1}\right| \leq \max_{1\leq j\leq n} \tilde{Y}^{(i)}_j - \min_{1\leq j\leq n} \tilde{Y}^{(i)}_j \\\leq \left(\max_{1\leq j\leq n} \mu^{(i)}_j - \min_{1\leq j\leq n} 
 \mu^{(i)}_j\right) + \left(\max_{1\leq j\leq n} \tilde{Z}^{(i)}_j - \min_{1\leq j\leq n} \tilde{Z}^{(i)}_j\right) \leq L_1 + 2 \max_{1\leq j\leq n}|\tilde{Z}^{(i)}_j|,\end{multline*}
 and therefore,
\[\EE{\left(\iso(\tilde{Y})_i -  \iso(\tilde{Y}^{(i)})_{m+1}\right)^2} \leq 2L_1^2 + 8\EE{ \max_{1\leq j\leq n}|\tilde{Z}^{(i)}_j|^2} \leq 2L_1^2 + 32\lambda^2\log n,\]
since the $\tilde{Z}^{(i)}_j$'s are Gaussian with variances bounded by $\lambda^2$, and the expected maximum of $n$ $\chi^2_1$'s is bounded by $4\log n$ for any $n\geq4$. 

Next we bound $\PP{\iso(\tilde{Y})_i\neq  \iso(\tilde{Y}^{(i)})_{m+1}}$.
By the property~\eqref{eqn:subsequence_equal} of isotonic regression, if $\iso(\tilde{Y})_i\neq  \iso(\tilde{Y}^{(i)})_{m+1}$ then it must be
the case that either $\iso(\tilde{Y})_i = \iso(\tilde{Y})_{i-m-1}$ or $\iso(\tilde{Y})_i = \iso(\tilde{Y})_{i+m+1}$.
Now, since $\mu$ is $L_0$-strictly increasing~\eqref{eqn:mu_L0}), we have
\[\big|\mu_i - \mu_{i-m-1}\big| \geq \frac{L_0(m+1)}{n}\]
and so, by the triangle inequality,
if $\iso(\tilde{Y})_i = \iso(\tilde{Y})_{i-m-1}$ then 
\[\max\left\{\big|\iso(\tilde{Y})_i - \mu_i\big|, \big|\iso(\tilde{Y})_{i-m-1} - \mu_{i-m-1}\big|\right\}
\geq \frac{L_0(m+1)}{2n} > \sqrt[3]{\frac{40L_1\lambda^2\log n}{n}},\]
where the last step holds by our definition of $m$.
Applying the same reasoning to the second case (i.e., $\iso(\tilde{Y})_i = \iso(\tilde{Y})_{i+m+1}$) and combining the two cases, 
this yields
\begin{equation}\label{eqn:step2_apply_yang}\PP{\iso(\tilde{Y})_i\neq \iso(\tilde{Y}^{(i)})_{m+1}} \leq \PP{\max_{j\in\{i-m-1,i,i+m+1\}}\big|\iso(\tilde{Y})_j - \mu_j\big|   > \sqrt[3]{\frac{40L_1\lambda^2\log n}{n}}}.\end{equation}
Now, we recall \citet{yang2018}'s result~\eqref{eqn:yang}---since $\tilde{Y}$ is $L_1$-Lipschitz with $\lambda$-subgaussian noise, choosing probability $\delta = 1/n^2$ 
we have
\begin{equation}\label{eqn:yang_for_upperbdthm}\PP{\max_{j_0\leq j\leq n+1-j_0}\big|\iso(\tilde{Y})_j - \mu^{(i)}_j| \leq \sqrt[3]{\frac{40L_1\lambda^2\log n}{n}}}\geq 1-1/n^2,\end{equation}
where
$j_0 = \left(\frac{\lambda n\sqrt{5\log n}}{L_1}\right)^{2/3}$.
In order to apply this to bound~\eqref{eqn:step2_apply_yang},
we need to ensure that $j_0 \leq i-m-1$ and $i+m+1\leq n+1-j_0$. Plugging in our definition of $m$, this is equivalent to
\begin{equation}\label{eqn:step2_i}
\min\{i,n+1-i\}\geq
\left\lceil\left(\frac{18\lambda n\sqrt{L_1\log n}}{L_0^{3/2}}\right)^{2/3}\right\rceil +  \left(\frac{\lambda n\sqrt{5\log n}}{L_1}\right)^{2/3}.\end{equation}
Combining~\eqref{eqn:yang_for_upperbdthm} with~\eqref{eqn:step2_apply_yang} we have
$\PP{\iso(\tilde{Y})_i\neq \iso(\tilde{Y}^{(i)})_{m+1}} \leq 1/n^2$.
Returning to our work above, therefore, we obtain
\begin{equation}\label{eqn:step2_bias}\left|\EE{\iso(\tilde{Y})_i} - \EE{\iso(\tilde{Y}^{(i)})_{m+1}}\right| \leq \frac{\sqrt{2L_1^2 + 32\lambda^2\log n }}{n}.\end{equation}

\subsubsection{Step 3: take a linear approximation}
Finally, we will
see that $\EE{\iso(\tilde{Y}^{(i)})_{m+1}}$ is not changed substantially if we replace it with a random vector
that has linear means and has constant variance.
Define
\[\check{\mu}_j = \frac{(i+m)-j}{2m}\cdot \mu_{i-m} + \frac{j-(i-m)}{2m}\cdot \mu_{i+m}, \quad i-m\leq j\leq i+m\]
and
\[\check{Z}_j = \frac{\sigma_i}{\sigma_j}\tilde{Z}_j, \quad i-m\leq j\leq i+m,\]
and finally let
\[\check{Y}^{(i)} = (\check{\mu}_{i-m} + \check{Z}_{i-m},\dots,\check{\mu}_i + \check{Z}_i,\dots,\check{\mu}_{i+m}+\check{Z}_{i+m}).\]
This construction is illustrated in Figure~\ref{fig:linearapprox}.

\begin{figure}[t]\centering
\includegraphics[width=0.6\textwidth]{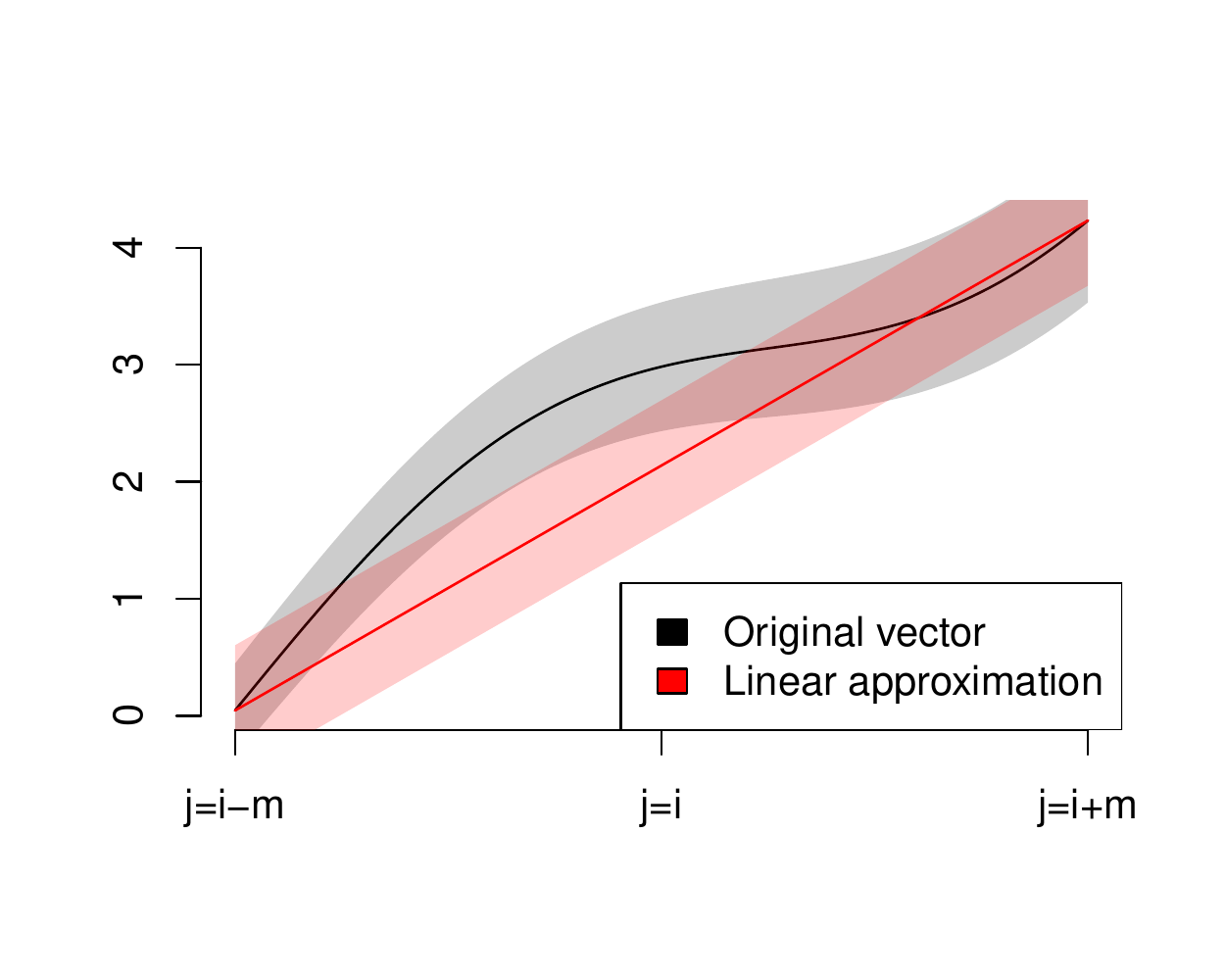}
\caption{Illustration of the linear approximation (Step 3 of the proof of Theorem~\ref{thm:upperbd_smooth},
with the distribution of $\tilde{Y}^{(i)}$ shown in black, and $\check{Y}^{(i)}$ in red. The means $\mu_j$ and $\check{\mu}_j$ are drawn
as solid black and red lines, respectively, while the standard deviations of $\tilde{Z}_j$ and $\check{Z}_j$ 
are represented by the black and red shaded regions.}
\label{fig:linearapprox}
\end{figure}

We will now show that $\EE{\left|\iso(\tilde{Y}^{(i)})_{m+1} - \iso(\check{Y}^{(i)})_i\right|}$ is small, so that we can use $\check{Y}^{(i)}$ to bound the bias of $\iso(\tilde{Y})_i$.
Using the ``min-max'' formulation of isotonic regression~\eqref{eqn:minmax},
\begin{align*}
\iso(\tilde{Y}^{(i)})_{m+1}
&= \max_{1\leq j\leq m+1}\min_{m+1\leq k\leq 2m+1} \overline{\tilde{Y}^{(i)}}_{j:k}\\
&\leq \max_{1\leq j\leq m+1}\min_{m+1\leq k\leq 2m+1} \overline{\check{Y}^{(i)}}_{j:k} + \max_{1\leq j\leq m+1\leq k\leq 2m+1}\left( \overline{\tilde{Y}^{(i)}}_{j:k} - \overline{\check{Y}^{(i)}}_{j:k}\right)\\
&=\iso(\check{Y}^{(i)})_{m+1}+ \max_{1\leq j\leq m+1\leq k\leq 2m+1}\left( \overline{\tilde{Y}^{(i)}}_{j:k}- \overline{\check{Y}^{(i)}}_{j:k}\right).
\end{align*}
An analogous lower bound holds similarly, and so
\begin{align*}
&\left|\EE{\iso(\tilde{Y}^{(i)})_{m+1}} - \EE{\iso(\check{Y}^{(i)})_{m+1}}\right|\\
&\leq \EE{ \max_{1\leq j\leq m+1\leq k\leq 2m+1}\left| \overline{\tilde{Y}^{(i)}}_{j:k} -\overline{\check{Y}^{(i)}}_{j:k}\right|}\\
&= \EE{ \max_{i-m\leq j\leq i\leq k\leq i+m}\left| \overline{(\mu + \tilde{Z})}_{j:k} -\overline{(\check{\mu}+\check{Z})}_{j:k}\right|}\textnormal{\quad by definition of $\tilde{Y}^{(i)}$ and $\check{Y}^{(i)}$}\\
&\leq  \max_{i-m\leq j\leq i\leq k\leq i+m}\left| \overline{\mu}_{j:k} -\overline{\check{\mu}}_{j:k}\right| +  \EE{ \max_{i-m\leq j\leq i\leq k\leq i+m}\left| \overline{\tilde{Z}}_{j:k} -\overline{\check{Z}}_{j:k}\right|}\\
&\leq  \max_{i-m\leq j\leq i+m}\left|\mu_j - \check{\mu}_j\right|+  \EE{ \max_{i-m\leq j\leq i\leq k\leq i+m}\left| \overline{\tilde{Z}}_{j:k} -\overline{\check{Z}}_{j:k}\right|}.
\end{align*}
Next,
since the random vector $\check{Y}^{(i)}$ is Gaussian with a linear mean and with constant variance,
by symmetry we can see that it has zero bias at its midpoint, i.e.,
\[\EE{\iso(\check{Y}^{(i)})_{m+1}} = \check{\mu}_i.\]
Therefore, by the triangle inequality,
\[\left|\EE{\iso(\tilde{Y}^{(i)})_{m+1}} - \mu_i\right|\leq 2 \max_{i-m\leq j\leq i+m}\left|\mu_j - \check{\mu}_j\right|+  \EE{ \max_{i-m\leq j\leq i\leq k\leq i+m}\left| \overline{\tilde{Z}}_{j:k} -\overline{\check{Z}}_{j:k}\right|}.\]
By definition of $\check{\mu}$ and  the smoothness assumption~\eqref{eqn:mu_smooth}, we have
\begin{multline*}
 \max_{i-m\leq j\leq i+m} \big|\mu_j - \check{\mu}_j\big|\\
 = \max_{i-m\leq j\leq i+m} \left|\mu_j - \left(\frac{(i+m)-j}{2m}\cdot \mu_{i-m} + \frac{j-(i-m)}{2m}\cdot \mu_{i+m}\right)\right|
\leq \frac{M}{4}\left(\frac{2m}{n}\right)^\beta,
\end{multline*}
 and therefore
\[\left|\EE{\iso(\tilde{Y}^{(i)})_{m+1}} - \mu_i\right|\leq \frac{M}{2}\left(\frac{2m}{n}\right)^\beta +  \EE{ \max_{i-m\leq j\leq i\leq k\leq i+m}\left| \overline{\tilde{Z}}_{j:k} -\overline{\check{Z}}_{j:k}\right|}.\]
Finally, to bound the last term, we have
\[\max_{i-m\leq j\leq i\leq k\leq i+m}\left|\overline{\tilde{Z}}_{j:k}-\overline{\check{Z}}_{j:k}\right|
= \max_{i-m\leq j\leq i\leq k\leq i+m}\left|\frac{1}{k-j+1}\sum_{\ell = j}^k \tilde{Z}_\ell \cdot \left(1 - \frac{\sigma_i}{\sigma_\ell}\right)\right|.\]
Since the $\tilde{Z}_j$'s are independent zero-mean Gaussians, for each pair of indices $j,k$, we see that $\frac{1}{k-j+1}\sum_{\ell = j}^k \tilde{Z}_\ell \cdot \left(1 - \frac{\sigma_i}{\sigma_\ell}\right)$
is a zero-mean Gaussian with standard deviation
\[\frac{1}{k-j+1}\sqrt{\sum_{\ell=j}^k \sigma_{\ell}^2\left(1 - \frac{\sigma_i}{\sigma_\ell}\right)^2} \leq \frac{1}{k-j+1}\sqrt{\sum_{\ell=j}^k \left(\frac{L_\sigma |\ell - i|}{n}\right)^2}\leq   \frac{L_\sigma\sqrt{m}}{n},\]
and so
\[\EE{\max_{i-m\leq j\leq i\leq k\leq i+m}\left|\overline{\tilde{Z}}_{j:k}-\overline{\check{Z}}_{j:k}\right|} \leq \frac{2 L_\sigma\sqrt{m \log n}}{n},\]
since there are at most $n^2/2$ many choices of $j,k$.
Combining everything, then,
\[\left|\EE{\iso(\tilde{Y}^{(i)})_{m+1}} - \mu_i\right| \leq \frac{M}{2}\left(\frac{2m}{n}\right)^\beta + \frac{2L_\sigma\sqrt{m\log n}}{n}.\]
Plugging in our choice of $m$, this simplifies to
\begin{equation}\label{eqn:step3_bias}\left|\EE{\iso(\tilde{Y}^{(i)})_{m+1}} - \mu_i\right|
\leq C_3  \left[\left(\frac{\log n}{n}\right)^{\beta/3} +  \left(\frac{\log n}{n}\right)^{2/3}\right],\end{equation}
for appropriately chosen $C_3$.

\subsubsection{Combining everything}
Combining our three steps, we see that the bounds~\eqref{eqn:step1_bias},~\eqref{eqn:step2_bias},~\eqref{eqn:step3_bias} combine to prove that
\[ \big|\EE{\iso(Y)_i} - \mu_i\big|\leq C\max\left\{\left(\frac{\log n}{n}\right)^{\beta/3} , \left(\frac{\log n}{n}\right)^{2/3},\frac{(\log n)^{10/3}}{n^{2/3}},\frac{\sqrt{\log n}}{n} \right\},\]
for $C$ chosen appropriately as a function of all the assumption parameters.
Since $\beta\leq 2$, the dominant term is $\left(\frac{\log n}{n}\right)^{\beta/3}$ for $\beta <2$ or $\frac{(\log n)^{10/3}}{n^{2/3}}$ for $\beta =2$.
 Examining the assumptions~\eqref{eqn:step1_i} and \eqref{eqn:step2_i} on the index $i$,
we see that this holds for all $i$ satisfying
\[C' (\log n)^{1/3}n^{2/3} \leq i \leq n+1 - C' (\log n)^{1/3}n^{2/3},\]
with $C'$ chosen appropriately as a function of all the assumption parameters. This completes the proof of the theorem.

We remark that,  if the noise is Gaussian, then Step 1 is not needed, and so the term $\frac{(\log n)^{10/3}}{n^{2/3}}$
does not appear in the upper bound. This means that the upper bound scales as $\left(\frac{\log n}{n}\right)^{\beta/3}$
both for the case $1\leq \beta <2$ and the case $\beta =2$, under Gaussian noise.

\subsection{Proof of lower bound for smoothness (Theorem~\ref{thm:lowerbd_smoothness})}\label{sec:proof_lowerbd_smooth}
Fix any $L_1 > L_0 \geq 0$, $M>0$, $\beta\in[1,2]$.
Consider a linear mean vector $\mu^{\textnormal{lin}}$, with entries
\[\mu^{\textnormal{lin}}_i = a_n \cdot \frac{i}{n},\]
and a mean vector $\mu$ that adds an oscillation,
\[\mu = \mu^{\textnormal{lin}} + b_n \Delta\textnormal{ where }\Delta_i =  \sin\left(c_n \cdot \frac{i}{n}\right).\]
This construction is illustrated in Figure~\ref{fig:lowerbd}.

We will specify the parameters $a_n,b_n,c_n\geq0$ later on, but for the moment
we assume that 
\begin{equation}\label{eqn:an_bn_cn} b_n c_n \leq a_n \textnormal{\quad and\quad }C_1 \leq c_n \leq C_2 n\textnormal{\quad and\quad }\frac{C_3}{\sqrt{n \log n}}\leq a_n \leq \frac{C_4(c_n)^{3/2}}{(\log n)^2\sqrt{n}},\end{equation}
where $C_1,C_2,C_3,C_4$ will not depend on $n$ and will be specified later on.
The first condition, $ a_n\geq b_n c_n$, ensures that $\mu$ is monotone nondecreasing. The second condition,  
essentially requiring that $1\ll c_n \ll n$, ensures that the oscillations of $\Delta$ are visible on the discrete points $i=1,\dots,n$---that is,
the sine wave completes many full cycles over the points $i=1,\dots,n$, and each cycle contains many indices $i$.
The third condition will be necessary for some calculations later on.

Now let $Z\sim\normal\big(\mathbf{0}_n,\sigma^2\mathbf{I}_n\big)$ for some $\sigma^2>0$, and 
define $Y = \mu+Z$ and $Y^{\textnormal{lin}} = \mu^{\textnormal{lin}} + Z$.
We will see that the nonlinear oscillations in $Y$ cannot be fully recovered by isotonic regression---while the 
gap between $Y$ and $Y^{\textnormal{lin}}$ is equal to $\pm b_n$ at the positive and negative peaks of the oscillation term $\Delta$,
the gap between $\iso(Y)$ and $\iso(Y^{\textnormal{lin}})$ at these points will typically be smaller. This is  because
$\iso(Y)_i$ and $\iso(Y^{\textnormal{lin}})_i$ are calculated via local averages (according to the min-max formula~\eqref{eqn:minmax}),
and the oscillations of $\Delta$ are therefore smoothed out, shrinking the difference between the two vectors towards zero.
The bias at or near the peaks of $\Delta$ will therefore be of order $b_n$, achieving a lower bound.

\begin{figure}[t]
\includegraphics[width=\textwidth]{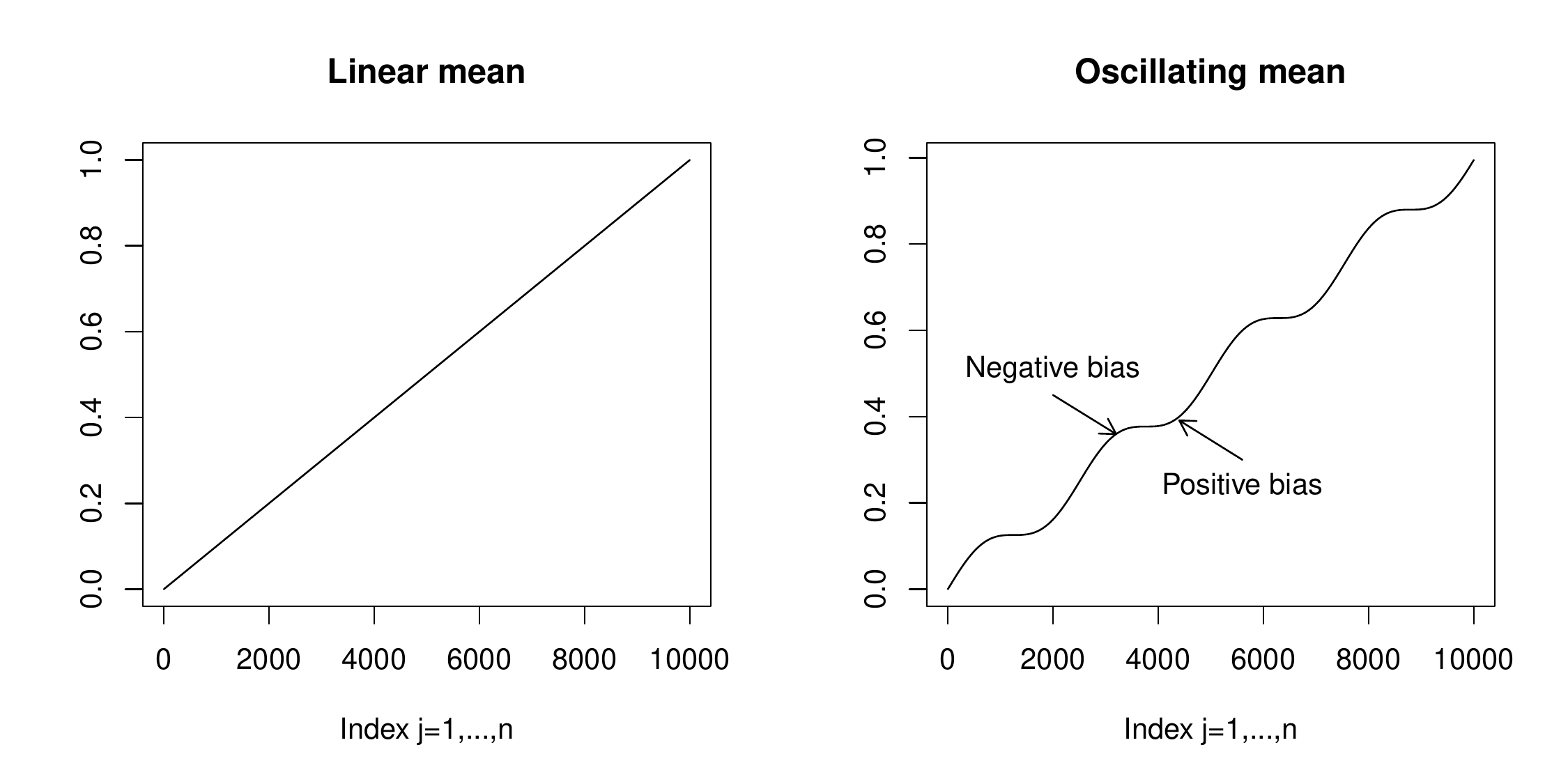}
\caption{Illustration of the construction for the lower bound result for smoothness (Theorem~\ref{thm:lowerbd_smoothness}).
The figure on the left illustrates the linear mean $\mu^{\textnormal{lin}}$, while the figure on the right illustrates the oscillating mean $\mu = \mu^{\textnormal{lin}}+b_n\Delta$.}
\label{fig:lowerbd}
\end{figure}

The remainder of the proof will follow these steps:
\begin{itemize}
\item {\bf Step 1:} We will show that
\begin{equation}\label{eqn:lowerbd_step_1}\mu_i \geq \mu^{\textnormal{lin}}_i + 0.8 b_n\textnormal{ for at least $C_5 n$ many indices $i$},\end{equation}
where $C_5$ will be specified later.
\item {\bf Step 2:} We will show that, for all indices $i$,
\[\iso(Y)_i \leq \iso(Y^{\textnormal{lin}})_i + b_n.\]
\item {\bf Step 3:} We will show that, for all indices $i$,
\[\textnormal{If $k_i(Y^{\textnormal{lin}})\geq i + \frac{C_6 n}{c_n}$ then }\iso(Y)_i \leq \iso(Y^{\textnormal{lin}})_i + 0.2 b_n ,\]
where $C_6$ will be specified later. 
\item {\bf Step 4:} We will show that
\[\PP{k_i(Y^{\textnormal{lin}})\geq i + \frac{C_6 n}{c_n}} \geq 0.5\textnormal{ for at least $(1-0.5C_5)n$ many indices $i$}.\]
\end{itemize}
Combining Steps 2, 3, and 4,  we see that
\begin{equation}\label{eqn:lowerbd_steps_234}\EE{\iso(Y)_i} \leq \EE{\iso(Y^{\textnormal{lin}})_i} + 0.6 b_n\textnormal{ for at least $(1-0.5C_5)n$ many indices $i$.}\end{equation}
Combining this with Step 1, therefore, there must be at least $0.5C_5n$ many indices $i$ for which 
the bounds~\eqref{eqn:lowerbd_step_1} and~\eqref{eqn:lowerbd_steps_234} both hold.
Applying the triangle inequality, we then have
\[\big|\EE{\iso(Y)_i} - \mu_i\big| + \big| \EE{\iso(Y^{\textnormal{lin}})_i} - \mu^{\textnormal{lin}}_i\big|\geq 0.2 b_n\]
for all such $i$.
This means that, for either $Y$ or $Y^{\textnormal{lin}}$,
it holds that the bias
satisfies the lower bound $0.1 b_n$ for at least $0.25C_5n$ many indices $i$.
By choosing $a_n,b_n,c_n$ appropriately, this will establish the lower bounds that we need.

We next give the details for the choice of $a_n,b_n,c_n$ to complete the proof of Theorem~\ref{thm:lowerbd_smoothness}, and then return to proving
the bounds in Steps 1, 2, 3, 4 above.
We will choose
\[a_n = \frac{L_1+L_0}{2} , \quad b_n = \min\left\{\frac{M}{2}, \frac{L_1 - L_0}{2} \right\} \cdot \big(n (\log n)^5\big)^{-\beta/3}, \quad c_n =n^{1/3}(\log n)^{5/3} .\]
First we check that the choices of $a_n,b_n,c_n$ satisfy the requirements~\eqref{eqn:an_bn_cn} for the steps of the proof,
which is trivial to verify.

Next, it's clear that $\mu,\mu^{\textnormal{lin}}$ are both $L_1$-Lipschitz~\eqref{eqn:mu_L1}
and $L_0$-strictly increasing~\eqref{eqn:mu_L0}. Finally we verify the H{\"o}lder smoothness condition~\eqref{eqn:mu_smooth}. For $\mu^{\textnormal{lin}}$ this is trivial
since it is a linear mean. For $\mu$,
we can write
\[\mu_i = f(i/n)\textnormal{ where } f(t) = a_n\cdot t + b_n \cdot \sin(c_n t),\]
for which we have
\[|f'(t_0)-f'(t_1)| = b_n c_n\cdot |\cos(c_n t_0) - \cos(c_n t_1)| \leq b_n c_n  \cdot \min\{c_n |t_0-t_1|,2\}.\]
Now we check that this is bounded by $M|t_0-t_1|^{\beta-1}$. If the first term in the minimum is larger, i.e., $|t_0-t_1|\geq 2/c_n$, then
\[M|t_0-t_1|^{\beta-1} \geq M(2/c_n)^{\beta-1} \geq 2b_nc_n,\]
by our choice of $b_n,c_n$. If instead the second term in the minimum is larger, i.e., $|t_0 - t_1|\leq 2/c_n$, then
\[M|t_0-t_1|^{\beta-1} = \frac{M|t_0-t_1|}{|t_0-t_1|^{2-\beta}} \geq \frac{M|t_0-t_1|}{(2/c_n)^{2-\beta}} \geq b_n(c_n)^2|t_0-t_1|,\]
by our choice of $b_n,c_n$. 
Therefore the function $f$ is $(\beta,M)$-H{\"o}lder smooth, and so this property is inherited by the vector $\mu$.
This verifies that $\mu,\mu^{\textnormal{lin}}$ each satisfy the assumptions needed for the theorem.

Applying the calculations above, we see that the bias of either $Y$ or $Y^{\textnormal{lin}}$
satisfies the lower bound $0.1 b_n$ for at least $0.25C_5n$ many indices $i$.
Since $b_n$ scales as $\big(n (\log n)^5\big)^{-\beta/3}$, this proves the desired result.

\subsubsection{Proof of Step 1}
The sine wave $\Delta$, given by $\Delta_i = \sin (c_n \cdot i/n)$, has period $\asymp n/c_n$, and we recall
that we have assumed that $C_1\leq c_n \leq C_2 n$.
This means that, for any $c\in (0,1)$, we can trivially show that, if $n$ is sufficiently large,
\[
\Delta_i \geq 1-c\textnormal{ for at least $c'n$ many indices $i\in\{1,\dots,n\}$}
\]
for some $c' >0$ that only depends on $c,C_1,C_2$ and not on $n$ or $c_n$.
Choosing $c = 0.2$, we then set $C_5 = c'$ to establish Step 1.

\subsubsection{Proof of Step 2}
By definition, we have
\[\big|Y_i -  Y^{\textnormal{lin}}_i \big| = \big|\mu_i - \mu^{\textnormal{lin}}_i\big| = b_n|\Delta_i|\\ 
\leq b_n\]
for all $i$,
and therefore, it holds deterministically that
\[\big|\iso(Y)_i - \iso(Y^{\textnormal{lin}})_i\big|\leq b_n\]
for all $i$ (this holds since $\|\iso(y)-\iso(y')\|_{\infty}\leq \|y-y'\|_{\infty}$ for any $y,y'\in\R^n$, by \citet[Lemma 1]{yang2018}).

\subsubsection{Proof of Step 3}

Next, we fix any index $i$, and consider the min-max formulation of isotonic regression as in~\eqref{eqn:minmax} and~\eqref{eqn:minmax_at_jk}:
\begin{align*}
\iso(Y)_i
&=\max_{j\leq i}\min_{k\geq i}\overline{Y}_{j:k}\\
&=\min_{k\geq i}\overline{Y}_{j_i(Y):k}\\
&=\min_{k\geq i}\left(\overline{Y^{\textnormal{lin}}}_{j_i(Y):k} + b_n \overline{\Delta}_{j_i(Y):k}\right)\\
&\leq\overline{Y^{\textnormal{lin}}}_{j_i(Y):k_i(Y^{\textnormal{lin}})} + b_n \overline{\Delta}_{j_i(Y):k_i(Y^{\textnormal{lin}})}\\
&\leq\max_{j\leq i}\overline{Y^{\textnormal{lin}}}_{j:k_i(Y^{\textnormal{lin}})} + b_n \overline{\Delta}_{j_i(Y):k_i(Y^{\textnormal{lin}})}\\
&=\iso(Y^{\textnormal{lin}})_i + b_n \overline{\Delta}_{j_i(Y):k_i(Y^{\textnormal{lin}})}.
\end{align*}
Therefore, it suffices to show that, for an appropriately chosen $C_6$,
\[\textnormal{If $k_i(Y^{\textnormal{lin}})\geq i + C_6 n/c_n$ then }\overline{\Delta}_{j_i(Y):k_i(Y^{\textnormal{lin}})}\leq 0.2 ,\]
Since $j_i(Y)\leq i$ by definition, we can weaken this to the statement
\[\overline{\Delta}_{j:k}\leq 0.2\textnormal{ for all indices $j,k\in\{1,\dots,n\}$ with $k\geq j+ C_6 n/c_n$.}\]
To see why this is true, observe that  $\Delta$ is a sine wave with period $\asymp n/c_n$.
Since the mean of the sine wave over one full cycle is zero, and the mean over any partial cycle is bounded by $1$,
this means that the bound above will hold for sufficiently large $C_6$ depending
only on $C_1,C_2$ and not on $n$ or $c_n$.

\subsubsection{Proof of Step 4} For this step we will use the breakpoint lemma. 
 We have
\[\sum_{j=\ell-m+1}^{\ell+m} (\mu^{\textnormal{lin}}_j - \mu^{\textnormal{lin}}_\ell)^2 \leq 2m\left(\frac{a_n m}{n}\right)^2\]
for any $m\geq 1$ and any $\ell$ with $m\leq \ell\leq n-m$. Choose $m = \left\lceil \left(\frac{n}{a_n \sqrt{\log n}}\right)^{2/3}\right\rceil$, and note
that $m\geq \frac{\log n}{C_2(C_4)^{2/3}}$ by our assumptions~\eqref{eqn:an_bn_cn}, which satisfies $m\geq 2$ for sufficiently large $n$.
Applying Lemma~\ref{lem:breakpointlemma} combined with the property~\eqref{eqn:subsequence_breakpoint} of isotonic regression, we have
\[\PP{\iso(Y^{\textnormal{lin}})_\ell\neq \iso(Y^{\textnormal{lin}})_{\ell+1}}\leq \frac{C_8 \log n}{m},\]
where $C_8$ depends only on $\sigma^2$. This implies that
\begin{multline*}\PP{k_i(Y^{\textnormal{lin}})< i + \frac{  m}{3 C_8\log n}} \\ = \PP{\iso(Y^{\textnormal{lin}})_\ell\neq \iso(Y^{\textnormal{lin}})_{\ell+1} \textnormal{ for some }i\leq \ell \leq i+  \frac{  m}{3 C_8\log n}}   \\
\leq \left( \frac{  m}{3 C_8\log n} + 1\right) \cdot   \frac{C_8 \log n}{m}\leq \frac{1}{3} + C_2(C_4)^{2/3}C_8\leq 0.5,\end{multline*}
as long as $i$ satisfies
\[m\leq i\leq n-m -  \frac{m}{3C_8 \log n}\]
and the constant $C_4$ in our assumption~\eqref{eqn:an_bn_cn} is chosen to be sufficiently small.

This completes Step 4 as long as 
\[\frac{C_6 n}{c_n} \leq \frac{  m}{3 C_8\log n} = \frac{1}{C_8\log n}\left\lceil \left(\frac{n}{a_n \sqrt{\log n}}\right)^{2/3}\right\rceil\]
and
\[
2m +  \frac{  m}{3 C_8\log n}  = \left(2 + \frac{1}{3+C_8\log n}\right) \left\lceil \left(\frac{n}{a_n \sqrt{\log n}}\right)^{2/3}\right\rceil \leq 0.5C_5n.\]
These conditions both hold for $a_n,c_n$ selected as in the proofs of the two theorems above (recalling condition~\eqref{eqn:an_bn_cn},
with $C_3$ chosen to be sufficiently large and $C_4$ sufficiently small).

\subsection{Proof of lower bound for strict monotonicity (Theorem~\ref{thm:lowerbd_strictincr})}\label{sec:proof_lowerbd_monotonicity}
Our argument for the proof of Theorem~\ref{thm:lowerbd_strictincr} is similar to that for the smoothness result, Theorem~\ref{thm:lowerbd_smoothness}.
Define function $f_0,f_1:[0,1]\rightarrow [0,1]$ as follows. Let $c_n,\epsilon_n>0$ be parameters that we will specify later
on. First let $g:[0,1]\rightarrow [0,1]$ be defined as
\[g(t) = \begin{cases}
0.5(2t)^\alpha, &0\leq t\leq 0.5,\\
1 - 0.5(2-2t)^\alpha, & 0.5\leq t \leq 1,\end{cases}\]
and then define
\[f_0 (t) = \begin{cases}
0, & 0\leq t \leq 0.5 -\epsilon_n,\\
c_n g\left(\frac{t - (0.5 - \epsilon_n)}{\epsilon_n}\right), & 0.5 - \epsilon_n \leq t \leq 0.5,\\
c_n, & 0.5 \leq t \leq 1,\end{cases}\]
and similarly
\[f_1 (t) = \begin{cases}
0, & 0\leq t \leq 0.5,\\
c_n g\left(\frac{t - 0.5}{\epsilon_n}\right), & 0.5 \leq t \leq 0.5+ \epsilon_n,\\
c_n, & 0.5+ \epsilon_n \leq t \leq 1.\end{cases}\]
Define $\mu^{(\ell)}_i = f_\ell(i/n)$ for each $\ell=0,1$ and each $i=1,\dots,n$. This construction is illustrated in Figure~\ref{fig:lowerbd2}.

We will now choose
\[c_n = C_1(n (\log n)^2)^{-\alpha/(2\alpha +1)}\text{ and }\epsilon_n = C_2(n (\log n)^2)^{-1/(2\alpha +1)},\]
where $C_1,C_2$ are constants depending only on the parameters $L_1,M,\alpha$.
We can verify that $f_0,f_1$ are both monotone nondecreasing, 
$L_1$-Lipschitz, and $(\beta,M)$-smooth with $\beta = \min\{2,\alpha\}$, as long as $C_1,C_2$ are chosen appropriately;
these properties are therefore inherited by the mean vectors $\mu^{(0)}$ and $\mu^{(1)}$.
Furthermore, $f_0,f_1$ both satisfy \citet{wright1981}'s condition~\eqref{eqn:alpha_wright} at $t_0 = 0.5$, specifically,
\[|f_\ell(t) - f_\ell(0.5) | \leq \frac{c_n 2^{\alpha-1}}{\epsilon_n^{\alpha}} |t - 0.5|^{\alpha} \leq C_3 |t - 0.5|^{\alpha}\text{ for each $\ell=0,1$},\]
where $C_3$ depends only on $C_1,C_2,\alpha$.

Now let $Z\sim\normal\big(\mathbf{0}_n,\sigma^2\mathbf{I}_n\big)$ for some $\sigma^2>0$, and 
define $Y^{(\ell)} = \mu^{(\ell)}+Z$ for each $\ell=0,1$.

We will see that the difference between $Y^{(0)}$ and $Y^{(1)}$ around the index $i_0=n/2$ cannot be fully recovered by isotonic regression.
The intuition for this phenomenon is the same as in the proof of Theorem~\ref{thm:lowerbd_smoothness};
since $\iso(Y^{(0)})_{i_0}$ and $\iso(Y^{(1)})_{i_0}$ are calculated via local averages near $i_0$ (according to the min-max formula~\eqref{eqn:minmax}),
while the difference between the vectors $Y^{(0)}$ and $Y^{(1)}$ is constrained to a vanishing region (they
are equal everywhere except for indices $i\in i_0 \pm n \epsilon_n)$, 
isotonic regression will often shrink the difference between the two signals at the index $i_0$---this will happen whenever the local average is
taken over a sufficiently wide range, i.e., in the min-max formula~\eqref{eqn:minmax}, $\iso(Y^{(\ell)})_{i_0} = \overline{Y^{(\ell)}}_{j:k}$ for $k-j$ sufficiently large.
The bias at index $i_0$ will therefore be of order $c_n$, achieving a lower bound.

\begin{figure}[t]
\includegraphics[width=0.5\textwidth]{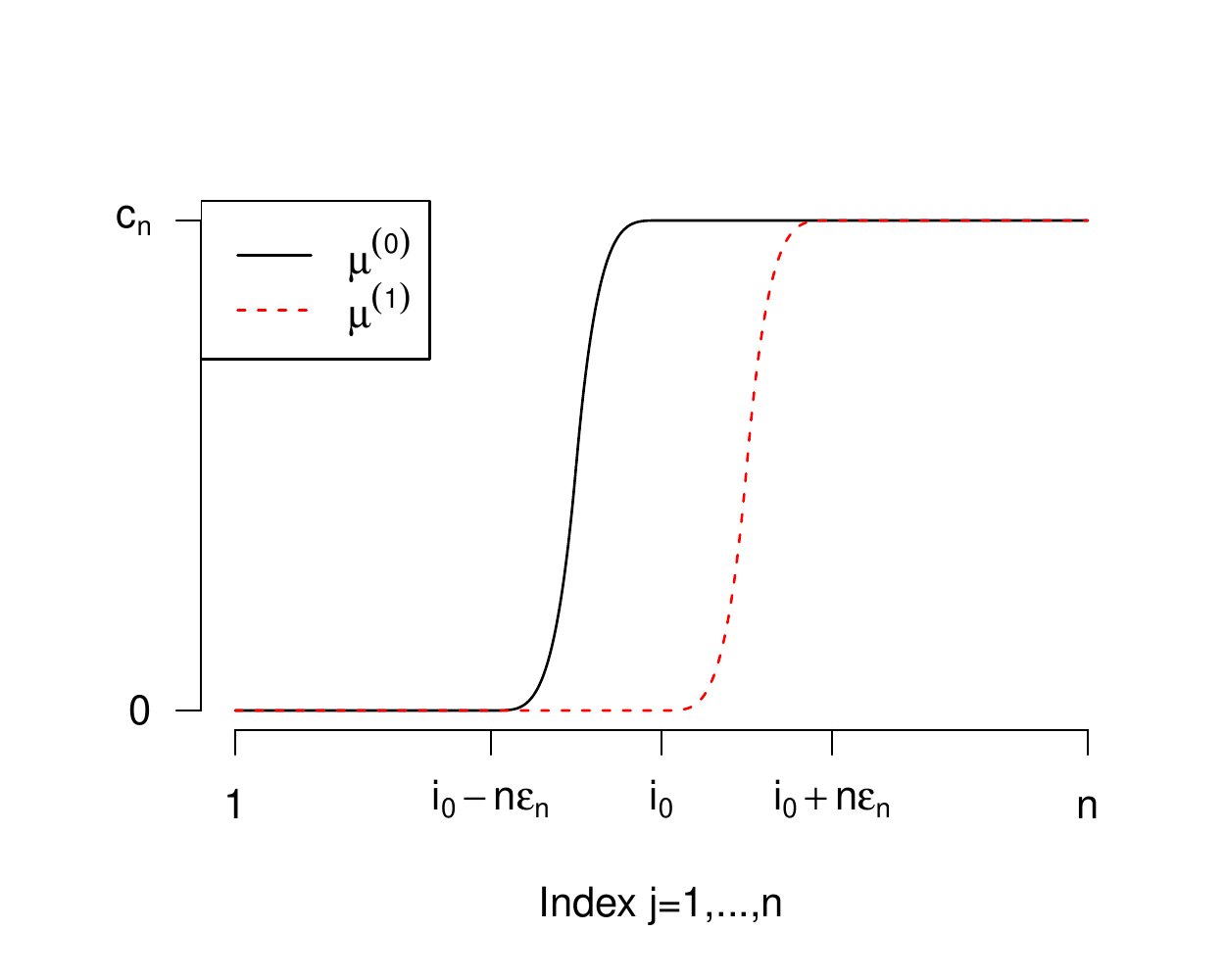}
\caption{Illustration of the construction for the lower bound result for strict monotonicity (Theorem~\ref{thm:lowerbd_strictincr}).
The figure illustrates the two different signal vectors $\mu^{(0)}$ and $\mu^{(1)}$ used in the construction for the proof.}
\label{fig:lowerbd2}
\end{figure}

The remainder of the proof will follow these steps:
\begin{itemize}
\item {\bf Step 1:} We will show that, deterministically,
\[\iso(Y^{(0)})_{i_0} \leq \iso(Y^{(1)})_{i_0} + c_n.\]
\item {\bf Step 2:} We will show that 
\[\textnormal{If $k_{i_0}(Y^{(1)})\geq i_0 + 10 n \epsilon_n$ then $\iso(Y^{(0)})_{i_0} \leq \iso(Y^{(1)})_{i_0} + 0.2 c_n$.}\]
\item {\bf Step 3:} We will show that
\[\PP{k_{i_0}(Y^{(1)})\geq  i_0 + 10 n \epsilon_n} \geq 0.5.\]
\end{itemize}
Combining these steps, we see that
\[\EE{\iso(Y^{(0)})_{i_0}} \leq \EE{\iso(Y^{(1)})_{i_0}} +0.6 c_n.\]
Applying the triangle inequality, we then have
\[\big|\EE{\iso(Y^{(0)})_{i_0}} - \mu^{(0)}_{i_0}\big| + \big| \EE{\iso(Y^{(1)})_{i_0}} - \mu^{(1)}_{i_0}\big|\geq 0.4 c_n\]
since $\mu^{(0)}_{i_0} = \mu^{(1)}_{i_0} + c_n$.
This means that, for either $Y^{(0)}$ or $Y^{(1)}$,
it holds that the bias at $i_0$
satisfies the lower bound $0.2 c_n$.
With $c_n$ as specified above, this completes the proof of Theorem~\ref{thm:lowerbd_strictincr}.

Now we turn to proving
the bounds in Steps 1, 2, 3 above.

\subsubsection{Proof of Step 1}
By definition, we have
\[\big|Y^{(0)}_i -  Y^{(1)}_i \big| = \big|\mu^{(0)}_i - \mu^{(1)}_i\big| 
\leq c_n\]
for all $i$,
and therefore, it holds deterministically that
\[\big|\iso(Y^{(0)})_i - \iso(Y^{(1)})_i\big|\leq c_n\]
for all $i$ (this holds since $\|\iso(y)-\iso(y')\|_{\infty}\leq \|y-y'\|_{\infty}$ for any $y,y'\in\R^n$, by \citet[Lemma 1]{yang2018}).

\subsubsection{Proof of Step 2}

Next, we define
\[\Delta = \frac{\mu^{(0)} - \mu^{(1)}}{c_n},\] and consider the min-max formulation of isotonic regression as in~\eqref{eqn:minmax} and~\eqref{eqn:minmax_at_jk}:
\begin{align*}
\iso(Y^{(0)})_{i_0}
&=\max_{j\leq i_0}\min_{k\geq i_0}\overline{Y^{(0)}}_{j:k}\\
&=\min_{k\geq i_0}\overline{Y}_{j_{i_0}(Y^{(0)}):k}\\
&=\min_{k\geq i_0}\left(\overline{Y^{(1)}}_{j_{i_0}(Y^{(0)}):k} + c_n \overline{\Delta}_{j_{i_0}(Y^{(0)}):k}\right)\\
&\leq\overline{Y^{(1)}}_{j_{i_0}(Y^{(0)}):k_{i_0}(Y^{(1)})} + c_n \overline{\Delta}_{j_{i_0}(Y^{(0)}):k_{i_0}(Y^{(1)})}\\
&\leq\max_{j\leq i}\overline{Y^{(1)}}_{j:k_{i_0}(Y^{(1)})} + c_n \overline{\Delta}_{j_{i_0}(Y^{(0)}):k_{i_0}(Y^{(1)})}\\
&=\iso(Y^{(1)})_{i_0} + c_n \overline{\Delta}_{j_{i_0}(Y^{(0)}):k_{i_0}(Y^{(1)})}.
\end{align*}
Therefore, it suffices to show that
\[\textnormal{If $k_{i_0}(Y^{(1)})\geq i_0 + 10 n\epsilon_n$ then }\overline{\Delta}_{j_{i_0}(Y^{(0)}):k_{i_0}(Y^{(1)})}\leq 0.2 ,\]
Since $j_{i_0}(Y^{(0)})\leq i_0$ by definition, we can weaken this to the statement
\[\overline{\Delta}_{j:k}\leq 0.2\textnormal{ for all indices $j,k\in\{1,\dots,n\}$ with $j \leq i_0$ and $k\geq i_0 + 10 n \epsilon_n$.}\]
This is true simply because, by definition of $\Delta$, we have $\Delta_i \leq 1$ for all $i$ and $\Delta_i = 0$ for all $|i- i_0| \geq n \epsilon_n$.

\subsubsection{Proof of Step 3} For this step we will use the breakpoint lemma. 
 We have
\[\sum_{j=\ell-m+1}^{\ell+m} (\mu^{(1)}_j - \mu^{(1)}_\ell)^2 \leq 2mc_n^2\]
for any $m\geq 1$ and any $\ell$ with $m\leq \ell\leq n-m$. Choose $m = \left\lceil \frac{C_4}{c_n^2\log n}\right\rceil$ which satisfies $m\geq 2$
for sufficiently large $n$.
Applying Lemma~\ref{lem:breakpointlemma}, we have
\[\PP{\iso(Y^{(1)})_\ell\neq \iso(Y^{(1)})_{\ell+1}}\leq \frac{C_5 \log n}{m},\]
where $C_5$ depends only on $\sigma^2,C_4$. This implies that
\begin{multline*}\PP{k_{i_0}(Y^{(1)})< i_0 + 10n\epsilon_n} \\ = \PP{\iso(Y^{(1)})_\ell\neq \iso(Y^{(1)})_{\ell+1} \textnormal{ for some }i_0\leq \ell \leq  i_0 + 10n\epsilon_n}   
\leq \left( 10n\epsilon_n + 1\right) \cdot   \frac{C_5 \log n}{m},\end{multline*}
as long as $i_0$ satisfies
\[m\leq i_0\leq n-m - 10n\epsilon_n.\]
By definition of $c_n,\epsilon_n,m$, as long as $C_1,C_2,C_4$ are chosen appropriately, we therefore have
\[\PP{k_{i_0}(Y^{(1)})< i_0 + 10n\epsilon_n} \leq 0.5,\]
which completes the proof of Step 3.

\section{Discussion}
In this work, we develop a sharp bound on the bias of isotonic regression in one dimension for strictly increasing signals,
establishing (up to log factors) that the bias is no larger than $n^{-2/3}$ for smooth signals, but may be as large as $n^{-1/3}$
in the non-smooth case. Many important open questions remain, for instance, whether these bounds on the bias
may be extended to the multidimensional setting, or to settings such as the Grenander estimator for a monotone density function.

\subsection*{Acknowledgements}
R.F.B.~was partially supported by the National Science Foundation via grant DMS-1654076 and by an Alfred P.~Sloan fellowship. G.R.~was partially supported by the National Science Foundation via grant DMS-1811767 and by the National Institute of Health via grant R01 GM131381-01. H. S.~was partially supported by the National Institute of Health via grant R01 GM131381-01. The authors thank Sabyasachi Chatterjee for helpful discussions.

\bibliographystyle{abbrvnat}
\bibliography{bib}

\appendix
\section{Additional proofs}\label{sec:app_proofs}

\subsection{Proof of breakpoint lemma (Lemma~\ref{lem:breakpointlemma})}\label{sec:proof_breakpointlemma}
In Section~\ref{sec:breakpointlemma}, we proved Lemma~\ref{lem:breakpoint_standardnormal}, which establishes
the desired result for the case of a standard Gaussian. We now need to reduce to this case.

First, we reduce to a shorter subvector of length $2m$. Define
\[\tilde{Y} = \big(Y_{i-m+1},Y_{i-m+2},\dots,Y_{i+m}\big).\]
The property~\eqref{eqn:subsequence_breakpoint} of isotonic regression implies that
\[\PP{\iso(Y)_i \neq \iso(Y)_{i+1}} \leq \PP{\iso(\tilde{Y})_m\neq \iso(\tilde{Y})_{m+1}},\]
so we now only need to bound this last probability.

Next we show that, since the means $\mu_j$ and variances $\sigma^2_j$ are nearly constant over $j=i-m+1,\dots,i+m$,
we can reduce to a standard Gaussian where the means and variances are constant.
We first state a trivial result comparing multivariate Gaussians, which we prove in Appendix~\ref{sec:proof_lem:reduce_to_constant}:

\begin{lemma}\label{lem:reduce_to_constant}
Fix any integer $k\geq 2$, and any $a_1,\dots,a_k\in\R$ and $b_1,\dots,b_k>0$. 
Let 
\[V\sim\normal\big((a_1,\dots,a_k),\diag\{b^2_1,\dots,b^2_k\}\big)\textnormal{\quad and\quad }W\sim\normal\big(\bar{a}\mathbf{1}_k,\bar{b}^2\mathbf{I}_k\big),\]
where
\[\bar{a} = \frac{1}{k}\sum_{i=1}^k a_i\textnormal{\quad and\quad } \bar{b}^2 = \frac{1}{k}\sum_{i=1}^k b^2_i.\]
Suppose that
\[\sum_{i=1}^k \left(\frac{a_i -\bar{a}}{\bar{b}}\right)^2 \leq \frac{C_1'}{\log k}\textnormal{\quad and\quad }
\max_{i=1,\dots,k}\left|\frac{b_i^2}{\bar{b}^2}-1\right| \leq \frac{C_2'}{\sqrt{k\log k}} .\]
Then, for any $c > 0$ and any measurable $\mathcal{A}\subseteq\R^k$,
\[\PP{V\in\mathcal{A}}\leq C_3' \cdot \Big( \PP{W\in\mathcal{A}} + k^{-c}\Big),\]
where $C_3'$ depends only on $c,C_1',C_2'$ and not on $k$.
\end{lemma}

Now we apply this result to $\tilde{Y}$. Let $\check{Y}\sim\normal\big(\bar{\mu}\mathbf{1}_{2m},\bar{\sigma}^2\mathbf{I}_{2m}\big)$,
with $\bar{\mu},\bar{\sigma}$ defined as in the statement of Lemma~\ref{lem:breakpointlemma}. Then applying Lemma~\ref{lem:reduce_to_constant}
with $k=2m$, $c=1$, and with
the set $\mathcal{A}$ defined as
\[\mathcal{A} = \left\{y\in\R^{2m} : \iso(y)_m \neq \iso(y)_{m+1}\right\},\]
we see that the conditions of Lemma~\ref{lem:reduce_to_constant} are satisfied for some $C_1',C_2'$ depending only on the values $C_1,C_2$ in the statement of Lemma~\ref{lem:breakpointlemma}.
So, we  have
\[\PP{\iso(\tilde{Y})_m\neq \iso(\tilde{Y})_{m+1}} \leq C_3'\left( \PP{\iso(\check{Y})_m \neq \iso(\check{Y})_{m+1}} + \frac{1}{m}\right),\]
where $C_3'$ depends only on $C_1,C_2$.

Finally, consider a standard multivariate Gaussian, $Z\sim\normal\big(\mathbf{0}_{2m},\mathbf{I}_{2m}\big)$. Clearly we can write $\check{Y} = \bar{\mu}\mathbf{1}_{2m} + \bar{\sigma}Z$,
and so
$\iso(\check{Y}) = \bar{\mu}\mathbf{1}_{2m} + \bar{\sigma}\cdot \iso(Z)$
by the property~\eqref{eqn:iso_locationscale} of isotonic regression. This implies that
\[\PP{\iso(\check{Y})_m \neq \iso(\check{Y})_{m+1}} = \PP{\iso(Z)_m \neq \iso(Z)_{m+1}} \leq \frac{\log m}{m-1},\] 
where the last step applies Lemma~\ref{lem:breakpoint_standardnormal}.

We have therefore proved that
\[\PP{\iso(\tilde{Y})_i\neq \iso(\tilde{Y})_{i+1}} \leq C_3'\left(\frac{\log m}{m-1} + \frac{1}{m}\right),\]
which completes the proof of Lemma~\ref{lem:breakpointlemma} with $C_3$ chosen appropriately.

\subsection{Proof of Gaussian coupling lemma (Lemma~\ref{lem:coupling})}\label{sec:coupling_proof}
Our lemma is a consequence of \citet{sakhanenko1985convergence}'s Gaussian coupling result:
\begin{theorem}[{\citet[Theorem 1]{sakhanenko1985convergence}}]\label{thm:sakhanenko}
Suppose that $Z_1,\dots,Z_n$ are independent random variables with $\EE{Z_j}=0$ and $\VV{Z_j}=\sigma^2_j$, and that for some $\nu>0$, each $Z_j$ satisfies
\[\EE{\nu |Z_j|^3 e^{\nu |Z_j|}}\leq \sigma^2_j.\]
Then there exists a coupling between $(Z_1,\dots,Z_n)$ and $(\tilde{Z}_1,\dots,\tilde{Z}_n)\sim \normal\big(0,\diag\{\sigma^2_1,\dots,\sigma^2_n\}\big)$ that satisfies
\[\EE{\exp\left\{c \nu \max_{1\leq j\leq n} \left|\sum_{k=1}^j Z_k - \sum_{k=1}^j \tilde{Z}_k \right|\right\}}\leq 1 + \nu\sqrt{\sum_{j=1}^n\sigma^2_j},\]
where $c>0$ is a universal constant.
\end{theorem}

Now define
$Z_j = Y_j - \mu_j$ for $j=1,\dots,n$.
To apply Theorem~\ref{thm:sakhanenko}, we will first work with the sequence $Z_i,Z_{i+1},\dots,Z_n$ instead of $Z_1,\dots,Z_n$, i.e.~we begin at the index $i$. Since each $Z_j$ is $(\lambda,\tau)$-subexponential, it therefore satisfies the assumption
$\EE{\nu |Z_j|^3 e^{\nu |Z_j|}}\leq \sigma^2_{\min} \leq \sigma^2_j$ when we choose $\nu$ as an appropriate function of $\lambda$, $\tau$, and $\sigma_{\min}$.
By Theorem~\ref{thm:sakhanenko}, then, we have a coupling between $(Z_i,\dots,Z_n)$ and $(\tilde{Z}_i,\dots,\tilde{Z}_n)\sim\mathcal{N}\big(0,\diag\{\sigma^2_i,\dots,\sigma^2_n\}\big)$ such that
\[\EE{\exp\left\{c \nu \max_{i\leq j\leq n} \left|\sum_{k=i}^j Z_k - \sum_{k=i}^j \tilde{Z}_k \right|\right\}}\leq 1 + \nu\sqrt{\sum_{j=i}^n\sigma^2_j}\leq 1 + \nu \lambda\sqrt{n}\]
(where we recall that $\max_j\sigma_j\leq\lambda$ by the subexponential tails assumption~\eqref{eqn:noise_assump} 
on the original noise terms $Z_j$).
In particular, this implies that
\begin{equation}\label{eqn:sakh_above_i} \PP{\max_{i\leq j\leq n}\left|\sum_{k=i}^j Z_k - \sum_{k=i}^j \tilde{Z}_k \right|
\leq C \log(n/\delta)}\geq 1-\delta\end{equation}
for all $\delta>0$, when $C$ is chosen appropriately as a function of $\lambda,\tau,\sigma_{\min}$.
Following an identical argument on the sequence $Z_{i-1},Z_{i-2},\dots,Z_1$, we can construct a coupling 
of $(Z_1,\dots,Z_{i-1})$ to a Gaussian vector $(\tilde{Z}_1,\dots,\tilde{Z}_{i-1})\sim \normal\big(0,\diag\{\sigma^2_1,\dots,\sigma^2_{i-1}\})\big)$ such that
\begin{equation}\label{eqn:sakh_below_i}\PP{\max_{1\leq j\leq i-1} \left|\sum_{k=j}^{i-1} Z_k- \sum_{k=j}^{i-1} \tilde{Z}_k \right|\leq C\log(n/\delta)}\geq 1-\delta.\end{equation}
Finally, since $(Z_1,\dots,Z_{i-1})$ and $(Z_i,\dots,Z_n)$ are independent, we can also take $(\tilde{Z}_1,\dots,\tilde{Z}_{i-1})$ and $(\tilde{Z}_i,\dots,\tilde{Z}_n)$ to be independent, and thus we have a coupling between $Z$ and a Gaussian random vector $\tilde{Z}\sim\normal\big(0,\diag\{\sigma^2_1,\dots,\sigma^2_n\}\big)$
such that both~\eqref{eqn:sakh_above_i} and~\eqref{eqn:sakh_below_i} hold, which implies that
\begin{equation}\label{eqn:sakh_across_i}\PP{\max_{1\leq j\leq i\leq k\leq n} \left|\sum_{\ell=j}^k Z_\ell - \sum_{\ell=j}^k \tilde{Z}_\ell \right|\leq  2C\log (n/\delta)}\geq 1-2\delta.\end{equation}

Let $\Delta = Z - \tilde{Z}$ denote the error in the coupling constructed above, and $\tilde{Y} = \mu + \tilde{Z}$. 
We therefore have $\overline{Y}_{j:k} = \overline{\tilde{Y}}_{j:k}+\overline{\Delta}_{j:k}$ for all indices $j,k$.
Now, for each $i$, let
\[j_i(\tilde{Y}) = \min\{j : \iso(\tilde{Y})_j = \iso(\tilde{Y})_i\}\textnormal{ and }k_i(\tilde{Y}) = \max\{k : \iso(\tilde{Y})_k = \iso(\tilde{Y})_i\}\]
as in~\eqref{eqn:minmax_at_jk},
which are the endpoints of the constant segment of the isotonic projection $\iso(\tilde{Y})$ containing index $i$.
Using the definition of isotonic regression, we calculate
\begin{multline*}
\iso(Y)_i
=\max_{j\leq i}\min_{k\geq i} \overline{Y}_{j:k}
\leq \max_{j\leq i} \overline{Y}_{j:k_i(\tilde{Y})}
\leq \max_{j\leq i} \overline{\tilde{Y}}_{j:k_i(\tilde{Y})} + \max_{j\leq i}\overline{\Delta}_{j:k_i(\tilde{Y})}\\
= \iso(\tilde{Y})_i + \max_{j\leq i}\overline{\Delta}_{j:k_i(\tilde{Y})}
\leq \iso(\tilde{Y})_i + \frac{1}{k_i(\tilde{Y}) - i+1 } \cdot \max_{1\leq j\leq i\leq k\leq n} \Big|\sum_{\ell=j}^k \Delta_\ell\Big|.
\end{multline*}
Similarly we can show that
\[\iso(Y)_i \geq \iso(\tilde{Y})_i - \frac{1}{i - j_i(\tilde{Y})+1 } \cdot \max_{1\leq j\leq i\leq k\leq n} \Big|\sum_{\ell=j}^k \Delta_\ell\Big|.\]
Therefore,
\begin{multline*}\EE{\left|\iso(Y)_i - \iso(\tilde{Y})_i\right|} \leq \EE{\left(\frac{1}{ i-j_i(\tilde{Y})+1}+\frac{1}{k_i(\tilde{Y}) - i+1}\right) \cdot \max_{1\leq j\leq i\leq k\leq n} \Big|\sum_{\ell=j}^k \Delta_\ell\Big|}\\
\leq  c'\log n \left(\EE{\frac{1}{i-j_i(\tilde{Y})+1}} + \EE{\frac{1}{k_i(\tilde{Y})-i+1}}\right) +  \EE{\left(\max_{1\leq j\leq i\leq k\leq n}\Big|\sum_{\ell=j}^k \Delta_\ell\Big| - c'\log n\right)_+}\end{multline*} 
for any $c'>0$.
We can further calculate
\begin{multline*}
\EE{\left(\max_{1\leq j\leq i\leq k\leq n}\Big|\sum_{\ell=j}^k \Delta_\ell\Big| - c'\log n\right)_+}
\leq \int_{t\geq c'\log n}\PP{\max_{1\leq j\leq i\leq k\leq n}\Big|\sum_{\ell=j}^k \Delta_\ell\Big| > t}\;\mathsf{d}t\\
\stackrel{\textnormal{ by~\eqref{eqn:sakh_across_i}}}{\leq} \int_{t\geq c'\log n} 2ne^{-t/2C}\;\mathsf{d}t
=4Cn e^{-(c'\log n)/2C}
\leq \frac{1}{n},
\end{multline*}
when $c'$ is chosen as an appropriate function of $C$.
Combining what we have so far, then,
\[\EE{\left|\iso(Y)_i - \iso(\tilde{Y})_i\right|} \leq   c'\log n \left(\EE{\frac{1}{i-j_i(\tilde{Y})+1}} + \EE{\frac{1}{k_i(\tilde{Y})-i+1}}\right) + \frac{1}{n}.\]
We will now bound these expected values. 

For any $\ell\geq 1$ with $i+\ell\leq n$, it is trivial to see that
\[\PP{k_i(\tilde{Y}) - i+1 = \ell} = \PP{k_i(\tilde{Y}) = \ell+i-1} \leq \PP{\iso(\tilde{Y})_{\ell+i-1}\neq \iso(\tilde{Y})_{\ell+i}}.\]
Therefore, for any $\ell_{\max}\geq 1$ with $i+\ell_{\max}\leq n$,
\begin{multline*} \EE{\frac{1}{k_i(\tilde{Y})-i+1}} \leq \sum_{\ell=1}^{\ell_{\max}} \frac{\PP{k_i(\tilde{Y})-i+1=\ell}}{\ell} + \frac{1}{\ell_{\max}} \\\leq  \sum_{\ell=1}^{\ell_{\max}} \frac{\PP{\iso(\tilde{Y})_{\ell+i-1}\neq \iso(\tilde{Y})_{\ell+i}}}{\ell} + \frac{1}{\ell_{\max}}.\end{multline*}
Next, we will use the breakpoint lemma (Lemma~\ref{lem:breakpointlemma}) to bound these probabilities. 
Fix
\[m = \ell_{\max}= \left\lceil \frac{n^{2/3}}{(\log n)^{1/3}}\right\rceil.\]
We apply the lemma at  the index $i'(\ell) = \ell+i-1$ in place of $i$. Since we've assumed that $\frac{C_2 n^{2/3}}{(\log n)^{1/3}}\leq i\leq n-\frac{C_2 n^{2/3}}{(\log n)^{1/3}}$,
it's trivial to check that $m \leq i'(\ell) \leq n-m$ for an appropriate choice of $C_2$. Next,
since the means $\mu_j$ are $L_1$-Lipschitz~\eqref{eqn:mu_L1} and the standard deviations $\sigma_j$ are $L_\sigma$-Lipschitz~\eqref{eqn:sig_Lip},
we can verify that the assumptions~\eqref{eqn:breakpointlemma_nearlyconstant} are satisfied 
for $C_1,C_2$ depending only on $L_1,L_\sigma$ (and not on $n$). Therefore
\[\PP{\iso(\tilde{Y})_{i'(\ell)} \neq \iso(\tilde{Y})_{i'(\ell)+1}}\leq \frac{C_3 (\log n)^{4/3}}{n^{2/3}}\]
for all $\ell= 1,\dots,\ell_{\max}$, where $C_3$ depends only on $C_1,C_2$. Returning to the work above, then,
\[ \EE{\frac{1}{k_i(\tilde{Y})-i+1}}  \leq  \sum_{\ell=1}^{\ell_{\max}} \frac{\frac{C (\log n)^{4/3}}{n^{2/3}}}{\ell} + \frac{1}{\ell_{\max}} \leq \frac{C'(\log n)^{7/3}}{n^{2/3}},\]
where $C'$ is chosen appropriately as a function of $C_3$. An identical argument holds for bounding $\EE{\frac{1}{i-j_i(\tilde{Y})+1}}$.
Combining everything, we see that
\[\EE{\left|\iso(Y)_i - \iso(\tilde{Y})_i\right|} \leq \frac{c'\cdot C'\cdot (\log n)^{10/3}}{n^{2/3}} + \frac{1}{n},\]
which proves the coupling lemma for an appropriately chosen $C_1$.

\subsection{Proof of Lemma~\ref{lem:reduce_to_constant}}\label{sec:proof_lem:reduce_to_constant}
First we define likelihoods,
\[f_V(x) = \frac{1}{(2\pi)^{k/2}\prod_{i=1}^k b_i}\exp\left\{-\sum_{i=1}^k (x_i - a_i)^2 / 2b_i^2\right\}\]
and
\[f_W(x) = \frac{1}{(2\pi)^{k/2} \bar{b}^k}\exp\left\{-\sum_{i=1}^k (x_i - \bar{a})^2 / 2\bar{b}^2\right\}.\]
Define the set
\[\mathcal{B} = \left\{x\in\R^k : f_V(x) > C_3' \cdot f_W(x)\right\},\]
where $C_3'$ will be defined below.
Then
\[\PP{V\in\mathcal{A}} \leq \PP{V\in\mathcal{B}} + \PP{V\in\mathcal{A}\backslash\mathcal{B}},\]
and
\begin{align*}
\PP{V\in\mathcal{A}\backslash\mathcal{B}}
&=\int_{x\in\R^k}\One{x\in\mathcal{A}\backslash\mathcal{B}} \cdot f_V(x)\;\mathsf{d}x\\
&=\int_{x\in\R^k}\One{x\in\mathcal{A}\backslash\mathcal{B}} \cdot \frac{f_V(x)}{f_W(x)} \cdot f_W(x)\;\mathsf{d}x\\
&\leq \int_{x\in\R^k}\One{x\in\mathcal{A}\backslash\mathcal{B}} \cdot C_3' \cdot f_W(x)\;\mathsf{d}x\\
&=C_3' \cdot \PP{W\in\mathcal{A}\backslash\mathcal{B}}\\
&\leq C_3'\cdot\PP{W\in\mathcal{A}},\end{align*}
where the first inequality holds by definition of $\mathcal{B}$.
Therefore, we now only need to bound $\PP{V\in\mathcal{B}}$.
We calculate
\begin{align*}
&\frac{f_V(V)}{f_W(V)}
 = \frac{\frac{1}{(2\pi)^{k/2}\prod_{i=1}^k b_i}\exp\left\{-\sum_{i=1}^k (V_i - a_i)^2 / 2b_i^2\right\}}{\frac{1}{(2\pi)^{k/2} \bar{b}^k}\exp\left\{-\sum_{i=1}^k (V_i - \bar{a})^2 / 2\bar{b}^2\right\}}\\
& = \underbrace{\left(\prod_{i=1}^k \frac{\bar{b}}{b_i}\right)}_{\textnormal{Term 1}}\cdot  \exp\left\{ \frac{1}{2}\underbrace{\sum_{i=1}^k \left(\frac{V_i - a_i}{b_i}\right)^2 \cdot \left(\frac{b^2_i}{\bar{b}^2}-1\right)}_{\textnormal{Term 2}}  + \underbrace{\sum_{i=1}^k \frac{V_i - a_i}{b_i}\cdot \frac{a_i - \bar{a}}{\bar b^2 / b_i}}_{\textnormal{Term 3}} + \frac{1}{2}\underbrace{\sum_{i=1}^k \left(\frac{a_i - \bar{a}}{\bar{b}}\right)^2}_{\textnormal{Term 4}}\right\}.
\end{align*}
Next we will bound each term separately. 

For Term 1, first note that we can assume $\frac{C_2'}{\sqrt{k\log k}} \leq \frac{1}{2}$ (because if this
does not hold then $k$ is bounded by a constant, and so the conclusion of the lemma holds trivially; i.e., by setting $C_3'$ appropriately, the claim reduces
to bounding a probability by $1$). Therefore, 
\begin{align*}
\textnormal{Term 1} &=\prod_{i=1}^k \frac{\bar{b}}{b_i}
=\exp\left\{-\frac{1}{2}\sum_{i=1}^k \log(b^2_i/\bar{b}^2)\right\}\\
&\leq \exp\left\{-\frac{1}{2}\sum_{i=1}^k \left[\big(b^2_i/\bar{b}^2-1\big) - 2\big(b^2_i/\bar{b}^2-1\big)^2\right]\right\}\textnormal{ since $|b^2_i/\bar{b}^2-1|\leq \frac{C_2'}{\sqrt{k\log k}}\leq \frac{1}{2}$}\\
&= \exp\left\{\sum_{i=1}^k \big(b^2_i/\bar{b}^2-1\big)^2 \right\}\textnormal{ by definition of $\bar{b}^2$}\\
&\leq \exp\left\{k\left(\max_{i=1,\dots,k}\left|\frac{b^2_i}{\bar{b}^2}-1\right|\right)^2\right\}
\leq \exp\left\{C_2'{}^2 / \log 2\right\},
\end{align*}
since $k\geq 2$. 

Next, for Term 2, note that $\left(\frac{V_i - a_i}{b_i}\right)^2\iidsim \chi^2_1$. Define
\[r^+_i = \max\left\{0, \left(\frac{b^2_i}{\bar{b}^2}-1\right)\right\} \textnormal{\quad and \quad } r^-_i = \max\left\{0,-\left(\frac{b^2_i}{\bar{b}^2}-1\right)\right\}.\]
 By \citet[Lemma 1]{laurent2000adaptive}, we have
\[\PP{\sum_{i=1}^k \left(\frac{V_i - a_i}{b_i}\right)^2 \cdot r^+_i   \leq \sum_{i=1}^k r^+_i + \max_{i=1,\dots,k} r^+_i \cdot \big(2\sqrt{kc\log k} + 2c\log k\big)}\geq 1-k^{-c}\]
and
\[\PP{\sum_{i=1}^k \left(\frac{V_i - a_i}{b_i}\right)^2 \cdot  r^-_i \geq \sum_{i=1}^k  r^-_i - \max_{i=1,\dots,k}  r^-_i \cdot 2\sqrt{kc\log k} }\geq 1-k^{-c}.\]
Combining the two, and using the fact that $\sum_{i=1}^k r^+_i - r^-_i = 0$ by definition of $\bar{b}$, this simplifies to
\[\PP{\textnormal{Term 2} \leq \max_{i=1,\dots,k} \left|\frac{b^2_i}{\bar{b}^2}-1\right|\cdot \big(4\sqrt{kc\log k} + 2c\log k\big)}\geq 1 - 2k^{-c}.\]
Since $\log k \leq \sqrt{k \log k}$ for all integers $k\geq 2$, we thus have
\[\PP{\textnormal{Term 2}  \leq  C_2' (4\sqrt{c} + 2c)}\geq 1-2k^{-c}.\]

Turning to Term 3, since $\frac{V_i - a_i}{b_i}\iidsim \normal(0,1)$, we have
\[\textnormal{Term 3} =  \sum_{i=1}^k \frac{V_i - a_i}{b_i}\cdot \frac{a_i - \bar{a}}{\bar b^2 / b_i}  \sim \normal\left(0, \sum_{i=1}^k \left(\frac{a_i - \bar{a}}{\bar b^2 / b_i}  \right)^2\right),\]
and so
\[\PP{\textnormal{Term 3} \leq \sqrt{\sum_{i=1}^k \left(\frac{a_i - \bar{a}}{\bar b^2 / b_i}  \right)^2} \cdot \sqrt{2c\log k}}\geq 1- k^{-c}.\]
We can weaken this to
\[\PP{\textnormal{Term 3} \leq \sqrt{1 + \max_{i=1,\dots,k}  \left|\frac{b^2_i}{\bar{b}^2}-1\right|}\sqrt{\sum_{i=1}^k \left(\frac{a_i - \bar{a}}{\bar b}  \right)^2} \cdot \sqrt{2c\log k}}\geq 1- k^{-c}.\]
Plugging in our definitions of $C_1',C_2'$, and using $k\geq 2$, we then have
\[\PP{\textnormal{Term 3} \leq \sqrt{1 + \frac{C_2'}{\sqrt{2 \log 2}}} \cdot  \sqrt{2cC_1'}}\geq 1- k^{-c}.\]
Finally, for the last term we can calculate
\[ \textnormal{Term 4} = \sum_{i=1}^k \left(\frac{a_i - \bar{a}}{\bar{b}}\right)^2 \leq \frac{C_1'}{\log k} \leq \frac{C_1'}{\log 2}.\]

Combining everything and simplifying, with probability at least $1-3k^{-c}$, we have
\[\frac{f_V(V)}{f_W(V)}\leq \exp\left\{\frac{C_2'{}^2}{\log 2}  +  C_2'(2\sqrt{c}+c) +\sqrt{1 + \frac{C_2'}{\sqrt{2 \log 2}}} \cdot  \sqrt{2cC_1'}+  \frac{C_1'}{2\log 2} \right\}.\]
Defining 
\[C_3' = \max\left\{3 , \exp\left\{\frac{C_2'{}^2}{\log 2}+  C_2'(2\sqrt{c}+c) +\sqrt{1 + \frac{C_2'}{\sqrt{2 \log 2}}} \cdot  \sqrt{2cC_1'}+  \frac{C_1'}{2\log 2} \right\}\right\},\]
we therefore have
\[\PP{V\in\mathcal{B}} = \PP{\frac{f_V(V)}{f_W(V)} >C_3'}  \leq 3k^{-c} \leq C_3'k^{-c},\]
as desired.

\end{document}